\documentclass[12pt,reqno]{amsart}
\usepackage{amsmath}
\usepackage{amssymb}
\usepackage{amstext}
\usepackage{mathrsfs}
\usepackage{a4wide}
\usepackage{graphicx}
\usepackage{bm}
\allowdisplaybreaks \numberwithin{equation}{section}
\usepackage{color}
\usepackage{cases}

\usepackage{hyperref}
\hypersetup{hypertex=true,
            colorlinks=true,
            linkcolor=blue,
            anchorcolor=blue,
            citecolor=blue }

\numberwithin{equation}{section}

\newtheorem{theorem}{Theorem}[section]
\newtheorem{proposition}[theorem]{Proposition}

\newtheorem{lemma}[theorem]{Lemma}

\theoremstyle{definition}

\newtheorem{definition}[theorem]{Definition}

\theoremstyle{remark}
\newtheorem{remark}[theorem]{Remark}

\def\d{\mathrm{d}}

\newcommand{\R}{\mathbb{R}}

\newcommand{\LL}{\mathcal{L}}
\newcommand{\Int}{\operatorname{int}}

\newcommand{\M}{\mathcal{M}}
\newcommand{\T}{\mathcal{T}}
\newcommand{\s}{\mathcal{S}(\R^2)}
\newcommand{\be}{\begin{equation}}
\newcommand{\ee}{\end{equation}}
\newcommand{\p}{\partial}

\begin{document}

\title[Radial symmetry of stationary and uniformly-rotating solutions]{Remarks on radial symmetry of stationary and uniformly-rotating solutions for the 2D Euler equation}

\author{Boquan Fan, Yuchen Wang, Weicheng Zhan}

\address{School of Mathematical Sciences and Institute of Natural Sciences, Shanghai Jiao Tong
University, 800 Dongchuan Road, Shanghai, P. R. China}
\email{fanboquan22@mails.ucas.ac.cn}

\address{School of Mathematics Science, Tianjin Normal University, Tianjin,  300387, P. R. China}
\email{wangyuchen@mail.nankai.edu.cn}

\address{School of Mathematical Sciences, Xiamen University, Xiamen, Fujian, 361005, P. R. China}
\email{zhanweicheng@amss.ac.cn}

\begin{abstract}
We prove that any uniformly rotating solution of the 2D incompressible Euler equation with compactly supported vorticity $\omega$ must be radially symmetric whenever its angular velocity satisfies $\Omega \in (-\infty,\inf \omega / 2] \cup \, [ \sup \omega / 2, +\infty )$, in both the patch and smooth settings. This result extends the rigidity theorems established in \cite{Gom2021MR4312192} (\textit{Duke Math. J.},170(13):2957-3038, 2021), which were confined to the case of non-positive angular velocities and non-negative vorticity. Moreover, our results do not impose any regularity conditions on the patch beyond requiring that its boundary consists of Jordan curves, thereby refining the previous result to encompass irregular vortex patches.
\end{abstract}

\maketitle


\bibliographystyle{plain}

\section{Introduction}

Consider the 2D incompressible Euler equation in the Euclidean plane $\R^2$. In the vorticity formulation the system is given by:
\begin{align}\label{1-1}
\begin{cases}
\partial_t\omega+\mathbf{v}\cdot \nabla\omega=0,\ (x,t)\in \R^2\times\mathbb{R}_+,&\\
\mathbf{v}=\nabla^\perp (-\Delta)^{-1}\omega, &\\
\omega|_{t=0}=\omega_0, &\\
\end{cases}
\end{align}
where $\nabla^\perp:= (\partial_2, -\partial_1)^T$ is the skew-gradient operator, $\mathbf{v}(x,t)= (v_1, v_2)^T$ is the velocity of the fluid and $\omega(x,t):=\partial_1 v_2 - \partial_2 v_1$ its vorticity. The second equation in \eqref{1-1}, commonly referred to as the \emph{Biot-Savart law}, recovers the velocity from the vorticity:
\[
\mathbf{v}(\cdot, t) = \nabla^\perp \left(\mathcal{N} \ast \omega(\cdot, t) \right),
\]
in which $$\mathcal{N}(x):=\frac{1}{2\pi}\ln \frac{1}{|x|}$$ is
 the Newtonian potential.

The global well-posedness of classical solutions to \eqref{1-1} was proved by Wolibner \cite{Wolibner1933}, see also  \cite{Marchioro1994}.  In his seminal work \cite{Yud1963MR158189}, Yudovich established the global well-posedness of weak solutions  $\omega \in L^\infty(\R^2) \,\cap \,L^1(\R^2)$.
As a consequence, since system \eqref{1-1} is considered as a non-local transport equation concerning the vorticity $\omega$, it follows that
the vorticity is only transported by the globally-in-time 2D Euler flows.
Specifically, Yudovich's theorem implies that a vortex patch solution of \eqref{1-1}, representing vorticity uniformly distributed in a bounded domain $D \subset \R^2$, i.e., $\omega_0 = 1_D$, with $1_D$ denoting the indicator function of domain $D$, retains the patch structure throughout its evolution. Consequently, the dynamics of a vortex patch $\omega=1_{D(t)}$ reduces to the boundary dynamics of its vortical domain:
\begin{equation} \label{E:1-2}
\partial_t x(\,\cdot\,,t) \cdot \vec{n}(x,t) = \mathbf{v}(x,t)\cdot \vec{n}(x,t),
\end{equation}
where the map $x(\,\cdot\,,t): \mathbb{T} \rightarrow \partial D(t)$ parameterizes the evolving boundary $\partial D(t)$, and $\vec{n}(x,t)$ denotes the outward unit normal to $\partial D(t)$ at $x$. 
The global regularity of vortex patch was established by Chemin \cite{Chemin1993}. Another elementary proof was provided by Bertozzi and Constantin \cite{Bertozzi1993}. For further details, we refer interested readers to  \cite{Bed2022MR4475666, Maj2002MR1867882} and references therein.

By equation \eqref{E:1-2}, $\omega=1_D$ is called as a  \emph{stationary patch solution} to \eqref{1-1} or a \emph{stationary vortex patch}, if
\[
\mathbf{v}(x)\cdot \vec{n}(x)=0 \text{ on } \partial D,
\]
meaning the patch remains unchanged over time.
Similarly, a vorticity distribution of the form  $\omega(x, t)=1_D(e^{-i\Omega t} x)$ is called a \emph{uniformly rotating vortex patch} with angular velocity $\Omega$ if $1_D$ becomes stationary in a frame rotating with angular velocity $\Omega$:
\begin{equation} \label{E:1-3}
\left(\nabla^\perp \left(\mathcal{N}\ast 1_D \right)+\Omega x^\perp\right)\cdot \vec{n}(x)=0 \text{ on } \partial D.
\end{equation}
Equation \eqref{E:1-3} gives rise to the integral equation
\begin{equation}\label{1-4}
	\mathcal{N}\ast 1_D+\frac{\Omega}{2}|x|^2\equiv C\ \ \ \text{on}\ \partial D,
\end{equation}
where the constant $C$ may vary across different connected components of $\partial D$. Obviously, uniformly rotating vortex patch recovers stationary vortex patch when $\Omega=0$, thus the stationary case is included in \eqref{1-4} also. It worths to noting that, compared with \eqref{E:1-3}, equation \eqref{1-4} is well-defined even $\p D$ is merely continuous. As a consequence, $\p D$ is allowed to consist of a union of ``rough'' Jordan curves, making the definition applicable to irregular vortex patches. Here the term ``irregular" refers to patch boundaries that are not $C^1$-smooth and may be highly singular, such as non-rectifiable or fractal sets.

Furthermore, equation \eqref{1-4} can be naturally extended to multi-patch solutions. Let $D_j$ $(j = 1, \dots, n)$ be bounded domains in $\R^2$, and let $\alpha_j$ $(j = 1, \dots, n)$ denote given constants.
We say that
\begin{equation}\label{1-4-0}
 \omega_0= \sum_{j=1}^n \alpha_j 1_{D_j}
\end{equation}
is a \emph{uniformly rotating multi-patch solution} for \eqref{1-1} with angular velocity $\Omega$ (which reduces to a stationary solution when $\Omega = 0$), if the following integral equation holds:
\begin{equation}\label{1-4-1}
   \mathcal{N}\ast \omega_0+\frac{\Omega}{2}|x|^2\equiv C_j\ \ \ \text{on}\ \partial D_j,\ j=1, \dots, n,
\end{equation}
where the constant $C_j$ is allowed to varing across different connected components of $\partial D_j$.

Analogously, we call $\omega(x, t) = \omega_0(e^{-i\Omega t} x)$ a \emph{uniformly rotating smooth solution} of \eqref{1-1} with angular velocity $\Omega$ (reducing to a stationary solution when $\Omega = 0$), provided the vorticity profile $\omega_0 \in C^2(\R^2)$ satisfies
\[
\left(\nabla^\perp \left(\mathcal{N}\ast \omega_0 \right) + \Omega x^\perp\right) \cdot \nabla \omega_0 = 0.
\]
This equation leads to the following integral equation
\begin{equation}\label{1-5}
   \mathcal{N}\ast \omega_0+\frac{\Omega}{2}|x|^2\equiv C\ \ \
\end{equation}
on each connected component of the regular level set $\{\omega_0=c\}$,
where the constant $C$ is allowed to vary between different connected components of the given regular level set.  \\

The main objective of this paper is to investigate the rigidity properties--the condition under which the solutions are necessarily radially symmetric--of uniformly rotating solutions to \eqref{1-1} with  compactly supported  \emph{sign-changing} vorticity, in both the patch and smooth settings. As noted on Page 7 of \cite{Gom2021MR4312192}, the authors mentioned that the requirement $\omega_0 \geq 0$ seems unnatural at first glance, and constructed non-radial patch solutions in the subsequent work \cite{Gom2024} if the vorticity is allowed to change its sign. We aim to show that the essential aspect of rigidity lies in the superharmonc/subharmonic nature of the stream function.

The radial symmetry properties of stationary simply-connected vortex patches have been established using the moving plane method; see \cite{Fra2000MR1751289}. Furthermore, under an additional convexity assumption on the patch, Hmidi \cite{Hmi2015MR3427065} provided an affirmative answer regarding the radial symmetry of rotating simply-connected patch solutions
when the angular velocity satisfies $\Omega \in (-\infty,0) \cup \{\frac{1}{2}\}$. In these works, the main idea is to reformulate
stationary/uniformly rotating 2D Euler flows as a globally defined semilinear elliptic equation in terms of the stream function $\psi$, where the vorticity is
connected with the stream function through a scalar function $f\in L^\infty(\R)$, known as the vorticity strength function, satisfying $\omega = f(\psi)$.
Consequently,
the radial symmetry of the solutions to \eqref{1-1} follows from  Liouville-type theorems for elliptic equations; see, for example, \cite{Gidas1979,Serra2013}. Moreover, this methodology has been  extensively applied to the study of border rigidity phenomena in the stationary Euler flows under certain structure assumptions \cite{Gui2024,Ham2019MR3951689,Hamel2023MR4556785,Ruiz2023,Wang2023}.

However, the correspondence between the vorticity strength function and the vorticity of steady 2D Euler flows is typically local, meaning that the vorticity strength function may not exist globally.
In the breakthrough work \cite{Gom2021MR4312192}, G\'{o}mez-Serrano, Park, Shi, and Yao established that a uniformly rotating vortex patch with angular velocity $\Omega \in(-\infty,0] \cup [1/2,+\infty)$ must exhibit radial symmetry, regardless of whether the domain is simply-connected, provided the patch boundaries are Lipschitz continuous Jordan curves. This smoothness requirement can be further relaxed to rectifiable Jordan curves, when the vortex patch is simply-connected and $\Omega \leq 0$. The bounds on the angular velocity $\Omega$ in this result are both sharp, as demonstrated by a series of studies on the bifurcations of vortex patches \cite{Hmidi2016,Hmi2013MR3054601}.

It is worth noting that the regularity of the patch boundary plays a crucial role in the analysis of \cite{Gom2021MR4312192}.  Although this assumption is mild and encompasses a large class of Jordan curves, it is significant to recognize that Jordan curves can exhibit highly complex geometries, including fractal structures such as the Koch snowflake. Indeed, a Jordan curve can even has a nonzero two-dimensional Lebesgue measure, see \cite{Sagan1994MR1299533}. These intricate examples lead to the interesting question of \emph{whether the radial symmetry properties established in \cite{Gom2021MR4312192} persist for vortex patch solutions with highly irregular boundaries}.

On the other hand, the radial symmetry properties for uniformly rotating multi-patch solutions and smooth solutions with compactly supported vorticity established in \cite{Gom2021MR4312192} rely on the assumption of
\emph{non-negative} vorticity and \emph{non-positive} angular velocities, cf. Corollary 2.10 and Theorem 2.12 in \cite{Gom2021MR4312192}.
This naturally raises the following question:

\emph{Under what conditions does a sign-changing uniformly rotating vorticity with negative angular velocity necessarily preserve radial symmetry?}

\noindent The approach employed in \cite{Gom2021MR4312192} seems to encounter challenges in dealing with sign-changing vorticity.

Motiviated by these questions, we revisit the radial symmetry issue of the stationary and uniformly rotating solutions on the plane, employing an alternative approach based on the \emph{continuous Steiner symmetrization} method and \emph{local symmetry} of the stream function it establishes.

Our first result concerning the patch solutions is given as follows:
\begin{theorem}\label{thm}
	Let $D_j$ $(j = 1, \dots, n)$ be bounded domains of $\R^2$, and let $\alpha_j$ $(j = 1, \dots, n)$ denote given constants. Assume that the boundary of each $D_j$ consists of a finite collection of mutually disjoint Jordan curves. Suppose $\omega_0=\sum_{j=1}^n \alpha_j 1_{D_j}$ is a uniformly rotating multi-patch solution to \eqref{1-1}, in the sense that it satisfies \eqref{1-4-1} for some $\Omega\in \R$. Then $\omega_0$ must be radially symmetric if $\Omega\le \inf \omega_0/2$ or $\Omega\ge \sup \omega_0/2$. Here, ``radially symmetric" should be interpreted as radially symmetric up to a translation when $\Omega = 0$.
\end{theorem}
\begin{remark}
When $n=1$ and $\alpha_1=1$, Theorem \ref{thm} refines the radial symmetry properties of uniformly rotating vortex patches established in Theorem A of \cite{Gom2021MR4312192}, where the patch boundaries are assumed to be smooth (in fact, rectifiability is sufficient, see Assumption (\textbf {HD}) therein), to include irregular vortex patches.
\end{remark}

\begin{remark}
The upper and lower bounds on the angular velocity given here are  sharp, see Appendix \ref{s5} for further details.
\end{remark}

The second result addresses on the radial symmetry properties of uniformly rotating smooth solutions of \eqref{1-1} with compactly supported vorticity, improving upon the previous result presented in Theorem B of \cite{Gom2021MR4312192}.
\begin{theorem}\label{thm2}
	Let $\omega_0 \in C^2(\R^2)$ have compact support. Assume $\omega(x, t)=\omega_0(e^{-i\Omega t} x)$ is a uniformly rotating smooth solution of \eqref{1-1}, in the sense that it satisfies \eqref{1-5} for some $\Omega\in \R$. Then $\omega_0$ must be radially symmetric if $\Omega\le \inf \omega_0/2$ or $\Omega\ge \sup \omega_0/2$. Here, ``radially symmetric" should be interpreted as radially symmetric up to a translation when $\Omega = 0$.
\end{theorem}

The main idea of the proofs is outlined as follows:  establishing the rigidity properties of steady 2D Euler flows is equivalent to demonstrating the radial symmetry of the stream function $\mathcal{N} \ast \omega_0 + \frac{\Omega}{2}|x|^2$. Furthermore, the condition on the angular velocity guarantees that the stream function is either superharmonic or subharmonic. Consequently, the task reduces to establishing the local symmetry of the stream function. To this end, we apply the continuous Steiner symmetrization, developed by Brock \cite{Bro1995MR1330619, Broc2000MR1758811, Bro2016MR3509374}, to derive the desired local symmetry properties. This idea has been employed in our recent work \cite{Fan2024}, where the rigidity properties of stationary and uniformly rotating patch and smooth solutions in the unit disk are established. Here the main challenges arise from the unboundedness of the domain, and the fact that the stream function satisfies the elliptic equation only piecewise. These difficulties are addressed by employing a carefully matched truncation argument and developing piecewise estimation techniques.

We would like to finish the introduction with some remarks on the {\it stationary wild weak solutions} of \eqref{1-1}, that is solutions with sufficiently low regularity. In a series of papers \cite{DeLellis2009,DeLellis2012,DeLellis2013}, De Lellis and Sz\'ekelyhidi developed the convex integration theory and the $h$-principle to demonstrate the non-uniqueness of wild weak solutions of the incompressible Euler equation. This framework was subsequently extended to including  the steady 2D incompressible Euler equation by Choffrut and Sz\'ekelyhidi
in \cite{Cho2014MR3505175},  where they constructed infintely many weak stationary solutions $\mathbf{v} \in L^\infty(\mathbb{T}^2)$ close to a given smooth steady Euler flow in the sense of the $H^{-1}$ distance.

Theorem \ref{thm} precludes the use of convex integration techniques to construct non-radial stationary rough vortex patches within the framework of the patch dynamics \eqref{E:1-2}. This result reveals that stationary structures of \eqref{E:1-2} are exceedingly rare near the Rankine vortex, even when irregular vortex patches are involved, in the absence of relative equilibria.

On the other hand, the role of vorticity profiles in this context remains unclear. An interesting problem is whether one can relax the smooth assumption in Theorem \ref{thm2}. This question is closely related to the applicability of definition \eqref{1-5} as the definition of stationary/uniformly rotating solutions of the 2D Euler flows. In particular, this equivalence is guaranteed when $\omega_0 \in C^2(\R^2)$ by Sard's theorem.

Another issue is whether the convex integration techniques can be employed to construct non-radial wild weak stationary/uniformly rotating solutions near a radially symmetric smooth vorticity profile. We believe that the methodology in \cite{Cho2014MR3505175} remains applicable in this setting.

The rest of the paper is organized as follows: In Section \ref{s2}, we present several preliminary results that will be employed in the subsequent analysis. Section \ref{s3} addresses the radial symmetry of uniformly rotating multi-patch solutions to \eqref{1-1} and provides the proof of Theorem \ref{thm}. Section \ref{s4} focuses on the radial symmetry of uniformly rotating smooth solutions and concludes with the proof of Theorem \ref{thm2}. In Appendix \ref{s5}, we consider bifurcation of concentric vortex patches then demonstrate that the bounds on the angular velocities in Theorem \ref{thm} are indeed sharp. \\

{\bf Notations:} Here we list notations and conventions  frequently used throughout this paper.
\begin{itemize}
  \item Let $B_r(x)$ and $Q_r(x)$ denote the open and closed discs in $\mathbb{R}^2$ with center $x$ and radius $r>0$, respectively. For convenience, we abbreviate $B_r:=B_r(0)$ and $Q_r:=Q_r(0)$.
  \item The non-negative part of a function $f$ is denoted by $f_+$, defined as $f_+=\max\{f, 0\}$.
  \item For a set $D$, we use $1_D$ to denote its indicator function.
  \item The terms ``infimum" and ``supremum" mentioned in this paper are understand in the sense of the essential infimum and essential supremum, respectively.
  \item $\mathcal{L}^N$ denotes $N$-dimensional Lebesgue measure.
  \item By $\mathcal{M}(\R^2)$ we denote the family of Lebesgue measurable sets in $\R^2$ with finite measure.
  \item For a function $u:\R^2\to \R$, let $\{u>a\}$ and $\{b\ge u>a\}$ denote the sets $\left\{x\in \R^2: u(x)>a \right\}$ and $\left\{x\in \R^2: b\ge u(x)>a \right\}$, respectively, ($a, b\in \R$, $a<b$).
  \item Let $\mathcal{S}(\R^2)$ be the set of real-valued measurable functions $u$ that satisfy
  \begin{equation*}
  	\LL^2(\{u>c\})<+\infty,\ \ \forall\, c>\inf u.
  \end{equation*}

\item For a Jordan curve $\Gamma$, let $\operatorname{int}(\Gamma)$ denote its interior, defined as the bounded connected component of $\mathbb{R}^2$ separated by $\Gamma$. By the Jordan curve theorem, $\operatorname{int}(\Gamma)$ is open and simply connected.
 \item  $C$ denotes a constant that is independent of the relevant parameters under consideration. The explicit value of $C$ may vary in different occurrences.
\end{itemize}

\section{Preliminaries}\label{s2}

We present some preliminary results involved in the proofs of the main results in this beginning section.

\subsection{The continuous Steiner symmetrization}
A key ingredient in our analysis is the continuous Steiner symmetrization \cite{Bro1995MR1330619, Broc2000MR1758811, Bro2016MR3509374, Sol2020MR4175494}. For the reader's convenience, and to make this paper self-contained, we recall its definition following the presentation in \cite{Bro2016MR3509374}. Although the original formulation applies to more general settings, here we slightly simplified the statement to cover our specific case well.

First, recall the definition of classical Steiner symmetrization, cf. \cite{Bro2007MR2569330, Kaw1985MR810619, Lie2001MR1817225}.
\begin{definition}[Steiner symmetrization]
\ \ \
  \begin{itemize}
    \item [(i)]For any set $M\in \M(\R)$, let
    \begin{equation*}
      M^*:=\left(-\frac{1}{2}\LL^1(M),\  \frac{1}{2}\LL^1(M)\right).
    \end{equation*}
    \item [(ii)]Let $M\in \M(\R^2)$. For every $x_2\in \R$, let
    \begin{equation*}
      M(x_2):=\left\{x_1\in \R: (x_1, x_2)\in M \right\}.
    \end{equation*}
    The set
    \begin{equation*}
      M^*:=\left\{ x=(x_1, x_2): x_1\in \left(M(x_2) \right)^*, x_2\in \R\right\}.
    \end{equation*}
    is called the Steiner symmetrization of $M$ (with respect to $x_1$).
    \item [(iii)]If $u\in \s$, then the function
    \begin{equation*}
      u^*(x):=\begin{cases}
                \sup \left\{c>\inf u: x\in \left\{u>c \right\}^*\right\}, & \mbox{if }\  x\in \bigcup_{c>\inf u} \left\{u>c \right\}^*, \\
                \inf u, & \mbox{if }\  x\not\in \bigcup_{c>\inf u} \left\{u>c \right\}^*,
              \end{cases}
    \end{equation*}
   is called the Steiner symmetrization of $u$ (with respect to $x_1$).
  \end{itemize}
\end{definition}

\begin{definition}[Continuous symmetrization of sets in $\M(\R)$]\label{Def2-2}
  A family of set transformations
  \begin{equation*}
    \T_t:\  \M(\R)\to \M(\R),\ \ \  0\le t\le +\infty,
  \end{equation*}
is called a continuous symmetrization on $\R$ if it satisfies the following properties: ($M, E\in \M(\R)$, $0\le s, t\le +\infty$)
\begin{itemize}
  \item [(i)]Equimeasurability property:\, $\LL^1(\T_t(M))=\LL^1(M)$,

    \smallskip
  \item [(ii)]Monotonicity property:\, If $M\subset E$, then $\T_t(M)\subset \T_t(E)$,

    \smallskip
  \item [(iii)]Semigroup property:\, $\T_t(\T_s(M))=\T_{s+t}(M)$,

    \smallskip
  \item [(iv)]Interval property:\, If $M$ is an interval $[x-R,\ x+R]$, ($x\in \R$, $R>0$), then $\T_t(M):=[xe^{-t}-R,\ xe^{-t}+R]$,

    \smallskip
  \item [(v)]Open/compact set property: If $M$ is open/compact, then $\T_t(M)$ is open/compact.
\end{itemize}
\end{definition}
\noindent For the construction of the family $\T_t$, $0 \le t \le +\infty$, we refer the reader to \cite[Theorem 2.1]{Broc2000MR1758811}.

Continuous Steiner symmetrization is to establishing a homotopy of classical Steiner symmetrization.
\begin{definition}[Continuous Steiner symmetrization (CStS)]\label{csts}
  \ \ \
  \begin{itemize}
    \item [(i)]Let $M\in \M(\R^2)$. The family of sets
    \begin{equation*}
      \T_t(M):=\left\{x=(x_1, x_2): x_1\in \T_t(M(x_2)), x_2\in \R \right\},\ \ \ 0\le t\le +\infty,
    \end{equation*}
    is called the continuous Steiner symmetrization (CStS) of $M$ (with respect to $x_1$).
    \item [(ii)]Let $u\in \s$. The family of functions $\T_t(u)$, $0\le t \le +\infty$, defined by
    \begin{equation}\label{2-0}
      \T_t(u)(x):=\begin{cases}
                \sup \left\{c>\inf u: x\in \T_t\left(\left\{u>c \right\}\right)\right\}, & \mbox{if }\ x\in \bigcup_{c>\inf u} \T_t\left(\left\{u>c \right\}\right), \\
                \inf u, & \mbox{if }\ x\not\in \bigcup_{c>\inf u} \T_t\left(\left\{u>c \right\}\right),
              \end{cases}
    \end{equation}
    is called CStS of $u$ (with respect to $x_1$).
  \end{itemize}
\end{definition}

For convenience, we will henceforth denote $M^t$ and $u^t$ as $\T_t(M)$ and $\T_t(u)$, respectively, for $t \in [0, +\infty]$. Note that if $u, v \in \mathcal{S}(\R^2)$ are continuous, nonnegative functions with compact support, and their supports satisfy $\text{supp}(u) \cap \text{supp}(v) = \varnothing$, then the property $(u+v)^t = u^t + v^t$ holds for any sufficiently small $t>0$.

\begin{remark}[Remark 2.5, \cite{Broc2000MR1758811}]
  Formula \eqref{2-0} is equivalent to the following relations:
  \begin{equation*}
    \begin{split}
       \{u^t>c\} & =\{u>c\}^t,\ \ \ \forall\,c>\inf u, \\
        \{u^t=\inf u\} & =\R^2\backslash\bigcup_{c>\inf u}\{u>c\}^t, \\
        \{u=+\infty\} & =\bigcap_{c>\inf u}\{u>c\}^t.
    \end{split}
  \end{equation*}
\end{remark}

We summarize below basic properties of CStS, established by Brock in \cite{Bro1995MR1330619, Broc2000MR1758811}.

\begin{proposition}\label{pro0}
  Let $M\in \M(\R^2)$, $u,v\in \s,\, t\in [0,+\infty]$. Then
  \begin{itemize}
    \item [(1)]Equimeasurability:
    \begin{equation*}
      \LL^2(M)=\LL^2(M^t)\ \ \ \text{and}\ \ \ \left\{u^t>c \right\}=\left\{u>c \right\}^t,\  \forall\, c>\inf u.
    \end{equation*}

    \smallskip
    \item [(2)]Monotonicity: If $u\le v$, then $u^t\le v^t$.

    \smallskip
    \item [(3)]Commutativity: If $\phi: \R \to \R$ is nondecreasing, then
    \begin{equation*}
      \phi(u^t)=[\phi(u)]^t.
    \end{equation*}

    \smallskip
    \item [(4)]Homotopy:
    \begin{equation*}
      M^0=M,\ \ \  u^0=u,\ \ \ M^\infty =M^*,\ \ \ u^\infty=u^*.
    \end{equation*}
Furthermore, from the construction of the CStS it follows that, if $M=M^*$ or $u=u^*$, then $M^t=M$, respectively, $u=u^t$ for all $t\in [0, +\infty]$.

    \smallskip
        \item [(5)]Cavalieri's pinciple: If $F$ is continuous and if $F(u)\in L^1(\R^2)$ then
        \begin{equation*}
          \int_{\R^2} F(u)\,\d x= \int_{\R^2} F(u^t)\,\d x.
        \end{equation*}

    \smallskip
    \item [(6)]Continuity in $L^p$: If $t_n\to t $ as $n\to +\infty$ and $u\in L^p(\R^2)$ for some $p\in [1, +\infty)$, then
    \begin{equation*}
      \lim_{n\to +\infty}\|u^{t_n}-u^t\|_p=0.
    \end{equation*}

    \smallskip
        \item [(7)]Nonexpansivity in $L^p$: If $u, v\in L^p(\R^2)$ for some $p\in [1, +\infty)$, then
        \begin{equation*}
          \|u^{t}-v^t\|_p\le \|u-v\|_p.
        \end{equation*}

          \smallskip
        \item [(8)]Hardy-Littlewood inequality: If $u, v\in L^2(\R^2)$ then
        \begin{equation*}
           \int_{\R^2} u^t v^t\,\d x\ge \int_{\R^2} u v\,\d x.
        \end{equation*}

         \smallskip
        \item [(9)]If $u$ is Lipschitz continuous with Lipschitz constant $L$, then $u^t$ is Lipschitz continuous, too, with Lipschitz constant less than or equal to $L$.

         \smallskip
        \item [(10)] If $\text{supp}\, (u)\subset B_R$ for some $R>0$, then we also have $\text{supp}\, (u^t)\subset B_R$. If, in addition, $u$ is Lipschitz continuous with Lipschitz constant $L$, then we have
            \begin{equation*}
              |u^t(x)-u(x)|\le LR\, t,\ \ \ \forall\, x\in B_R.
            \end{equation*}
           Furthermore, there holds
           \begin{equation*}
             \int_{B_R}G(|\nabla u^t|)\,\d x \le  \int_{B_R}G(|\nabla u|)\,\d x,
           \end{equation*}
           for every convex function $G: [0, +\infty) \to [0, +\infty)$ with $G(0)=0$.
  \end{itemize}
\end{proposition}

\subsection{Local symmmetry}
Following Brock \cite{Broc2000MR1758811}, we introduce a local version of symmetry for a function $u\in \s$.
\begin{definition}[Local symmetry in a certain direction]\label{def2-1}
Let $u\in \s$ be a function with compact support. Suppose the set $$\{x: 0< u(x) < \sup_{\R^2}u \}$$ is open, and that $u$ is continuously differentiable on this set. The function $u$ is said to be \emph{locally symmetric in the direction $x_1$} if  the following condition holds:

\noindent If $y=(y_1, y_2)\in \R^2$ satisfies
  \begin{equation*}
    0<u(y)<\sup_{\R^2} u,\ \ \ \partial_1 u(y)>0,
  \end{equation*}
  and $\tilde{ y}$ is the (unique) point satisfying
  \begin{equation*}
    \tilde{y}=(\tilde{y}_1, y_2),\ \ \ \tilde{y}_1>y_1,\ \ \ u(y)=u(\tilde{y})< u(s, y_2),\ \ \forall\, s\in (y_1,\tilde{y}_1),
  \end{equation*}
  then
  \begin{equation*}
  \partial_2 u(y)=\partial_2u(\tilde{y}),\ \ \ \partial_1u(y)=-\partial_1 u (\tilde{y}).
  \end{equation*}
\end{definition}
\vspace{0.3 cm}
Suppose that for arbitrary rotations $x\mapsto y=(y_1, y_2)$ of the coordinate system, $u$ is locally symmetric in the direction $y_1$. Then $u$ is said to be \emph{locally symmetric}. In other words, a function $u$ is said to be locally symmetric if it is locally symmetric in \emph{every} direction.

Although locally symmetric functions are not globally radial, they possess strong symmetric characteristics. Roughly speaking, it is radially symmetric and radially decreasing in some annuli (probably infinitely many) and flat elsewhere.
\begin{proposition}[ \cite{Broc2000MR1758811}, Theorem 6.1]\label{pro1}
Let $u \in \s$ be a locally symmetric function. Set
\[
V:=\{x \in \R^2: 0<u(x) <\sup_{\R^2} u \}.
\]
Then we have the following decomposition:
    \begin{itemize}
    \item [(1)]$\displaystyle V=\bigcup_{k\in K} A_k \cup \{x\in V: \nabla u(x)=0\}$, \text{where}
    \begin{equation*}
      A_k=B_{R_k}(z_k)\backslash {Q_{r_k}(z_k)},\ \ \ z_k\in \R^2,\ \ \ 0\le r_k<R_k;
    \end{equation*}
     \item[(2)]$K$ is a countable set;
        \smallskip
    \item [(3)] the sets $A_k$ are pairwise disjoint;
    \smallskip
     \item [(4)]$u(x)=U_k(|x-z_k|)$, $x\in A_k$, where $U_k\in C^1([r_k, R_k])$;
         \smallskip
    \item [(5)]$U'_k(r)<0$ for $r\in (r_k, R_k)$;
        \smallskip
    \item [(6)]$u(x)\ge U_k(r_k),\ \forall\,x\in Q_{r_k}(z_k)$, $k\in K$.
  \end{itemize}
\end{proposition}

It can be seen that if $u\in C^1(\overline{B_R})$ ($u\not \equiv 0$) for some $R>0$ is locally symmetric when extended by zero outside $B_R$ and viewed as a function on $\mathbb{R}^2$, then the super-level sets $\{u>t\}$ $(t\ge 0)$ are countable unions of mutually disjoint discs, and $|\nabla u|=\text{const.}$ on the boundary of each of these discs.

The following result demonstrates that local symmetry, when combined with superharmonicity, leads to global symmetry.

\begin{lemma}\label{key0}
  Let $R > 0$, let $a\in \R^2$, and let $u \in C^1(\overline{B_R(a)})$ satisfy $u > 0$ in $B_R(a)$ and $u = 0$ on $\partial B_R(a)$. Suppose that $u$ is locally symmetric. Suppose also that $u$ is a weakly superharmonic function in the sense that
    \begin{equation*}
\int_{B_R(a)} \nabla u \cdot \nabla \varphi \, \d x \geq 0
\end{equation*}
    for all non-negative functions $\varphi \in C_c^\infty(B_R(a))$. Then $u$ is a radial function with respect to $a$, namely, $u=u(|x-a|)$. Moreover, there exists $r_0\in[0, R)$ such that $u$ is constant in $B_r(a)$, and for $r=|x-a|\in (r_0, R)$, we have $\frac{\d u}{\d r}<0$.
    \end{lemma}

\begin{proof}
   Let $A_k=B_{R_k}(z_k)\backslash {Q_{r_k}(z_k)}\subset B_R(a)$ denote an annulus in the decomposition specified in Proposition \ref{pro1}. It holds that $u(x)=U_k(|x-z_k|)$, $x\in A_k$. Note that $U_k'(r_k)=0$ and $U_k'(R_k)=0$ if $R_k<R$; see also \cite[Remark 2.2]{Bro2022MR4375744}. We claim that $R_k=R$, and consequently, $u$ is a radially decreasing function in $B_R(a)\backslash Q_{r_k}(a)$. If the assertion would not hold, then $U_k'(R_k)=0$. Notice that $u=c$ on $\partial B_{R_k}(z_k)$ for some constant $c\ge 0$. Since $u$ is weakly superharmonic, by the Hopf lemma, we have that $U_k'(R_k)<0$. This leads to a contradiction. Therefore, we have $R_k=R$, and thus $z_k=a$. In light of $U_k'(r_k)=0$, by a similar argument, we conclude that if $r_k>0$, then $u$ must be constant in $B_{r_k}(a)$. The proof is thus complete.
   \end{proof}

\subsection{A symmetry criterion due to F. Brock}

The following symmetry criterion is due to Brock \cite{Broc2000MR1758811}.
\begin{proposition}[\cite{Broc2000MR1758811}, Theorem 6.2]\label{pro2}
Let $u\in H^1(\R^2)\cap C(\R^2)$ be a nonnegative function with compact support. Suppose that $u$ is continuously differentiable on $V$, where
\begin{equation*}
  V=\left\{x\in \R^2: 0<u(x)<\sup u\right\}.
\end{equation*}
Then, if
  \begin{equation*}
    \lim_{t\to 0}\frac{1}{t}\left(\int_{B_R} |\nabla u|^2\,\d x-\int_{B_R} |\nabla u^t|^2\,\d x\right)=0,
  \end{equation*}
  $u$ is locally symmetric in the direction $x_1$.
\end{proposition}

\subsection{Some auxiliary lemmas}
We now provide three useful estimates that will play a crucial role in the subsequent proofs of the main results. Recalling Definition \ref{csts}, for a given function $u\in \s$, we denote by $u^t$ the CStS of $u$ with respect to $x_1$.

\begin{lemma}\label{key1}
Let $u\in C^1(\R^2)$ satisfy $u(x)\to -\infty$ as $|x|\to \infty$.  Let $c_0$ and $c_1$ be two constants such that $c_0>c_1$. Then
\begin{equation*}
 \int_{\R^2}|\nabla (u-c_0)_+^t|^2\d x-\int_{\R^2}|\nabla (u-c_0)_+|^2\d x\ge \int_{\R^2}|\nabla (u-c_1)_+^t|^2\d x-\int_{\R^2}|\nabla (u-c_1)_+|^2\d x
\end{equation*}
for any $t\ge0$.
\end{lemma}

\begin{proof}
  In view of Proposition \ref{pro0} (3), we observe that $(u-c_0)_+^t=(u^t-c_0)_+$. Thus, we have
  \begin{equation*}
    \int_{\R^2}|\nabla (u-c_0)_+^t|^2\d x-\int_{\R^2}|\nabla (u-c_0)_+|^2\d x=\int_{\{u^t>c_0\}}|\nabla u^t|^2\d x-\int_{\{u>c_0\}}|\nabla u|^2\d x.
  \end{equation*}
  Similarly,
  \begin{equation*}
    \int_{\R^2}|\nabla (u-c_1)_+^t|^2\d x-\int_{\R^2}|\nabla (u-c_1)_+|^2\d x=\int_{\{u^t>c_1\}}|\nabla u^t|^2\d x-\int_{\{u>c_1\}}|\nabla u|^2\d x.
  \end{equation*}
  We now estimate
  \begin{equation*}
  \begin{split}
      &  \int_{\R^2}|\nabla (u-c_0)_+^t|^2\d x-\int_{\R^2}|\nabla (u-c_0)_+|^2\d x \\
       & = \int_{\{u^t>c_0\}}|\nabla u^t|^2\d x-\int_{\{u>c_0\}}|\nabla u|^2\d x \\
       & = \left(\int_{\{u^t>c_1\}}|\nabla u^t|^2\d x-\int_{\{u>c_1\}}|\nabla u|^2\d x\right)- \left(\int_{\{c_0\ge u^t>c_1\}}|\nabla u^t|^2\d x-\int_{\{c_0\ge u>c_1\}}|\nabla u|^2\d x\right)\\
       &= \int_{\R^2}|\nabla (u-c_1)_+^t|^2\d x-\int_{\R^2}|\nabla (u-c_1)_+|^2\d x\\
       &\ \ \ \ \ \ - \left(\int_{\{c_0\ge u^t>c_1\}}|\nabla u^t|^2\d x-\int_{\{c_0\ge u>c_1\}}|\nabla u|^2\d x\right)
  \end{split}
  \end{equation*}
  So it suffices to prove that
  \begin{equation}\label{2-2}
   \int_{\{c_0\ge u^t>c_1\}}|\nabla u^t|^2\d x-\int_{\{c_0\ge u>c_1\}}|\nabla u|^2\d x\le 0.
  \end{equation}
  Consider a nondecreasing function $\phi: \R \to \R$ defined by
  \begin{equation*}
    \phi(s)=\begin{cases}
              c_0-c_1, & \mbox{if}\  s\ge c_0, \\
              s-c_1, & \mbox{if}\  c_1<s<c_0, \\
              0, & \mbox{if}\  s\le c_1.
            \end{cases}
  \end{equation*}
It is clear that $\phi(u)$ has compact support, as $u(x) \to -\infty$ when $|x| \to \infty$. By the definition of $\phi$, we first have
  \begin{equation}\label{2-3}
     \int_{\R^2} |\nabla \phi(u)|^2\d x = \int_{\{c_0\ge u>c_1\}}|\nabla u|^2\d x.
  \end{equation}
    By Proposition \ref{pro0} (3), we observe that $[\phi(u)]^t=\phi(u^t)$. Hence, it follows that
\begin{equation}\label{2-4}
  \int_{\R^2} |\nabla [\phi(u)]^t|^2\d x  = \int_{\R^2} |\nabla \phi(u^t)|^2\d x=\int_{\{c_0\ge u^t>c_1\}}|\nabla u^t|^2\d x.
\end{equation}
On the other hand, Proposition \ref{pro0} (10) implies that
  \begin{equation}\label{2-5}
    \int_{\R^2} |\nabla [\phi(u)]^t|^2\d x\le \int_{\R^2} |\nabla \phi(u)|^2\d x.
  \end{equation}
Now, \eqref{2-2} follows directly from the combination of \eqref{2-3}, \eqref{2-4}, and \eqref{2-5}. The proof is thus complete.
\end{proof}

\begin{lemma}\label{key2}
Let $u: \R^2 \to \R$ is a nonnegative, Lipschitz continuous function with compact support. Let $V$ be a measurable subset of $\R^2$. Suppose that $u\equiv c$ in $V$ for some constant $c>0$. Then
  \begin{equation*}
    \int_{V}|u^t-u|\,\d x=o(t)
  \end{equation*}
  as $t\to0$.
\end{lemma}

\begin{proof}
Let $\phi_1(s)=1_{(c,\infty)}(s)$ and $\phi_2(s)=1_{[c,\infty)}(s)$ be two nondecreasing functions. In view of Proposition \ref{pro0} (3) and (6), we have $\phi_j(u^t)=[\phi_j(u)]^t$, and $\phi_j(u^t)\to \phi_j(u)$ in $L^1(\R^2)$ as $t\to 0$ for $j=1, 2$. It follows that $\LL^2(V\cap\{u^t\not= c\})\to 0$ as $t\to 0$. By Proposition \ref{pro0} (10), it follows that $\|u^t-u\|_{L^\infty(\R^2)}\le Ct$ for some constant $C>0$ independent of $t$. Therefore, we conclude that
\begin{equation*}
 \left| \int_{V}(u^t-u)\,\d x \right|\le \int_{V}|u^t-u|\,\d x\le Ct \LL^2(V\cap\{u^t\not= c\})=o(t),
\end{equation*}
from which the desired result follows. The proof is thus complete.
\end{proof}

\begin{lemma}\label{key3}
  Let $u: \R^2 \to \R$ is a nonnegative, Lipschitz continuous function with compact support. Let $\Gamma \subset \R^2$ be a Jordan curve, and let $V=\operatorname{int}(\Gamma)$.

  Assume that $u$ is continuously differentiable and weakly superharmonic in $V$. Assume further that $u=c$ on $\Gamma$ for some constant $c\in \R$ and $u > c$ in $V$. Let ${\gamma_j}$ be a sequence of decreasing regular values of $u$ satisfying $\gamma_j > c$ for $j = 1, 2, \dots$, and $\lim_{j \to \infty} \gamma_j = c$. Then
  \begin{equation}\label{2-22}
    \int_{V}(u^t-u)\,\d x=o(t)
  \end{equation}
  as $t\to0$.
\end{lemma}

\begin{proof}
To establish \eqref{2-22}, it is enough to demonstrate that, for any $\varepsilon > 0$,
  \begin{equation}\label{2-23}
    \limsup_{t\to 0}\frac{\left|\int_{V}(u^t-u)\,\d x \right|}{t}\le \varepsilon.
  \end{equation}
   In view of Proposition \ref{pro0} (10), we have $\|u^t-u\|_{L^\infty(\R^2)}\le Ct$ for some constant $C>0$ independent of $t$. Since $\gamma_j$ is a regular value of $u$, the level set $u^{-1}(\{\gamma_j\})\cap V$ consists of a finite union of $C^1$ simple closed curves, denoted by $\{\ell_i\}_{i=1}^{n_j}$. Set $G_i=\operatorname{int}(\ell_i)$. Observe that $u$ is weakly superharmonic in $V$. Hence,
  \begin{equation*}
    V\backslash\left(\cup_{i=1}^{n_j}\overline{G_i}\right)=\{x\in V\mid c<u(x)<\gamma_j\}
  \end{equation*}
  which implies that $\LL^2\left(V\backslash\left(\cup_{i=1}^{n_j}\overline{G_i}\right)\right)\to 0$ as $j\to \infty$. Fix a sufficiently large integer $k$ such that
  \begin{equation}\label{2-23'}
    C\LL^2\left(V\backslash\left(\cup_{i=1}^{n_{k}}\overline{G_i}\right)\right)<\varepsilon.
  \end{equation}
   We compute
  \begin{equation}\label{2-24}
    \int_V (u^t-u)\,\d x=\sum_{i=1}^{n_{k}}\int_{G_i}(u^t-u)\,\d x+\int_{V\backslash\left(\cup_{i=1}^{n_k}\overline{G_i}\right)}(u^t-u)\,\d x.
  \end{equation}
Now let us show that
\begin{equation}\label{2-25}
  \int_{G_i}(u^t-u)\,\d x=o(t),\ \ \ i=1,\dots, n_k,
\end{equation}
as $t\to0$. We can choose pairwise disjoint domains $U_i$, $i=1,\dots, n_k$, such that $\overline{G_i}\subset U_i\subset V$ and $\{u>\gamma_k\}\cap U_i=G_i$ for each $i=1,\dots, n_k$. By the definition of the CStS, we observe that $(u^t-\gamma_k)_+1_{U_i}$ is a rearrangement of $(u-\gamma_k)_+1_{U_i}$ for all sufficiently small $t$. Thus, for all sufficiently small $t$,
\begin{equation*}
 \int_{U_i}(u^t-\gamma_k)_+\,\d x=\int_{U_i}(u-\gamma_k)_+\,\d x,
\end{equation*}
which yields that
\begin{equation}\label{2-26}
\begin{split}
   &\int_{G_i}(u^t-\gamma_k)_+\,\d x-\int_{G_i}(u-\gamma_k)_+\,\d x \\
   & = \left[\int_{U_i}(u^t-\gamma_k)_+\,\d x- \int_{U_i}(u-\gamma_k)_+\,\d x\right]- \int_{U_i\backslash {G_i}}(u^t-\gamma_k)_+\,\d x \\
     & =-\int_{U_i\backslash {G_i}}\left[(u^t-\gamma_k)_+-(u-\gamma_k)_+ \right]\,\d x.
\end{split}
\end{equation}
Let $f(s):=1_{(\gamma_{k}, \infty)}(s)$. By Proposition \ref{pro0} (3) and (6), we have $f(u^t)=[f(u)]^t$, and $f(u^t)\to f(u)$ in $L^1(\R^2)$ as $t\to 0$. It follows that $\LL^2\left((U_i\backslash {G_i}\right)\cap \{u^t>\gamma_k\})\to 0$ as $t\to 0$. Hence, we have
\begin{equation}\label{2-27}
\begin{split}
      \left| \int_{U_i\backslash {G_i}}\left[(u^t-\gamma_k)_+-(u-\gamma_k)_+ \right]\,\d x\right|&\le \int_{U_i\backslash {G_i}}\left|(u^t-\gamma_k)_+-(u-\gamma_k)_+ \right|\,\d x \\
     & \le Ct \LL^2\left((U_i\backslash {G_i})\cap \{u^t>\gamma_k\}\right)=o(t).
\end{split}
\end{equation}
By combining \eqref{2-26} and \eqref{2-27}, we obtain
\begin{equation*}
\begin{split}
    o(t) &  =\int_{G_i}(u^t-\gamma_k)_+\,\d x-\int_{G_i}(u-\gamma_k)_+\,\d x \\
     & =\int_{G_i}\left(\int_{0}^{1}f(u(x)+\theta (u^t(x)-u(x)))\d \theta   \right)(u^t(x)-u(x))\,  \d x.
\end{split}
\end{equation*}
Therefore, we have
\begin{equation}\label{2-210}
\begin{split}
\int_{G_i}&(u^t-u)\,\d x=\int_{G_i}f(u(x))(u^t(x)-u(x))\,\d x\\
      =\int_{G_i}& \left(f(u(x))- \int_{0}^{1}f(u(x)+\theta (u^t(x)-u(x)))\d \theta  \right)(u^t(x)-u(x))\,  \d x+o(t).
\end{split}
\end{equation}
Recalling that $\|u^t-u\|_{L^\infty(\R^2)}\le Ct$ and the monotonicity of $f$, it follows that
\begin{equation}\label{2-211}
\begin{split}
    &  \left|\int_{G_i} \left(f(u(x))- \int_{0}^{1}f(u(x)+\theta (u^t(x)-u(x)))\d \theta  \right)(u^t(x)-u(x)) \, \d x \right|\\
     & \le Ct  \int_{G_i} \left | f(u(x))- \int_{0}^{1}f(u(x)+\theta (u^t(x)-u(x)))\d \theta  \right|  \d x \\
     & \le Ct \int_{G_i} |f(u(x))-f(u^t(x))|\,\d x\le Ct\|f(u^t)-f(u)\|_{L^1(\R^2)}=o(t).
\end{split}
\end{equation}
Thus, \eqref{2-25} follows directly from the combination of \eqref{2-210} and \eqref{2-211}. In view of \eqref{2-24}, we conclude that
\begin{equation*}
   \int_V (u^t-u)\,\d x=o(t)+\int_{V\backslash\left(\cup_{i=1}^{n_k}\overline{G_i}\right)}(u^t-u)\,\d x,
\end{equation*}
which yields that
\begin{equation*}
    \limsup_{t\to 0}\frac{\left|\int_{V}(u^t-u)\,\d x \right|}{t} \le \limsup_{t\to 0}\frac{\left|\int_{V\backslash\left(\cup_{i=1}^{n_k}\overline{G_i}\right)}(u^t-u)\,\d x \right|}{t} \le  C\LL^2\left(V\backslash\left(\cup_{i=1}^{n_{k}}\overline{G_i}\right)\right)\le \varepsilon,
\end{equation*}
where \eqref{2-23'} was used in the final inequality. Thus, the proof of \eqref{2-23} is complete, concluding the proof of the lemma.

\end{proof}

\section{Radial symmetry of uniformly rotating multi-patch solutions}\label{s3}

In this section, we establish the radial symmetry of uniformly rotating multi-patch solutions to \eqref{1-1}. Let $D_j$ $(j = 1, \dots, n)$ be bounded domains of $\R^2$, and let $\alpha_j$ $(j = 1, \dots, n)$ denote given constants. Assume that the boundary of each $D_j$ is composed of a finite collection of mutually disjoint Jordan curves. Assume $\omega_0=\sum_{j=1}^n \alpha_j 1_{D_j}$ is a stationary multi-patch solution \eqref{1-1}, in the sense that it satisfies \eqref{1-4-1} for some $\Omega\in \R$. Our goal is to prove that $\omega_0$ is radially symmetric if $\Omega\le \inf \omega_0/2$ or $\Omega\ge \sup \omega_0/2$. Here, the term ``radially symmetric" is to be understood as radially symmetric, up to a translation, in the stationary case (i.e., $\Omega = 0$).

Note that the desired radial symmetry will be obtained if we can demonstrate that the stream function $u = \mathcal{N} \ast \omega_0+\frac{\Omega}{2}|x|^2$ is radial, up to a translation in $\R^2$. We now proceed to demonstrate that $u$ possesses this property.

We denote the outer boundary of $D_j$ by $\Gamma_0^{(j)}$, and the inner boundaries by $\Gamma^{(j)}_i$ for $i = 1, \dots, m_j$. Set $V_i^{(j)}=\operatorname{int}(\Gamma^{(j)}_i)$ for $i=0, \dots, m_j$. Under the given assumptions, there exist constants $c^{(j)}_i$ $(j=1, \dots, n;\ i=0, \dots, m_j)$ such that
\begin{align}\label{3-1}
\begin{cases}
-\Delta u=\sum_{j=1}^n \alpha_j 1_{D_j}-2\Omega \ &  \text{in}\ \R^2,\\
u=c^{(j)}_i\ &\text{on}\ \Gamma^{(j)}_i,\ \text{for}\ j=1, \dots, n;\ i=0, \dots, m_j.
\end{cases}
\end{align}
For clarity, we will consider the two cases $\Omega \leq \inf \omega_0 / 2$ and $\Omega \geq \sup \omega_0 / 2$ separately.
\subsection{The case $\Omega\le \inf \omega_0/2$}\label{subsec3-1}
In this case, the stream function $u = \mathcal{N} \ast \omega_0+\frac{\Omega}{2}|x|^2$ is weakly superharmonic. Note that $\omega$ has compact support. It also follows that $\Omega \le 0$ and $u(x)\to -\infty$ as $|x|\to \infty$. Furthermore, we have
\begin{equation*}
  \nabla u(x)=\frac{1}{2\pi}\int_{\R^2}\frac{x-y}{|x-y|^2}\omega_0(y)\,\d y+\Omega\, x=\Omega\, x+\frac{1}{2\pi}\int_{\R^2}\omega_0(y)\,\d y\cdot \frac{x}{|x|^2}+O\left(\frac{1}{|x|^2}\right)
\end{equation*}
as $|x|\to \infty$. By the implicit function theorem, there exists $\delta_0 <\min \{c^{(j)}_i: j=1, \dots, n;\ i=0, \dots, m_j\}$ such that, for all $c < \delta_0$, the level set $\Im_c:=\{ u = c \}$ is a simple, closed, $C^1$ curve enclosing the origin. By the maximum principle, we have $\{ u > c_0 \} = \Int(\Im_{c_0})$. We will demonstrate that $(u - c)_+$ is radial, up to a translation, for any $c < \delta_0$, from which the desired result follows. We now fix such a $c_0<\delta_0$. Our goal is to establish that $(u - c_0)_+$ is a radial function up to a translation in $\mathbb{R}^2$. The result will be proved if we can show that $(u - c_0)_+$ is locally symmetric. Indeed, if $(u - c_0)_+$ is locally symmetric, then $\{u>c_0\}$ are countable unions of mutually disjoint disks. Since $\{u>c_0\}= \Int(\Im_{c_0})$ is connected, it follows that $\{u>c_0\}$ is a disk, denoted by $B_{R}(a)$, and hence $\Im_{c_0}=\partial B_R(a)$. Notice that $(u-c_0)_+$ is locally symmetric and weakly superharmonic in $B_{R}(a)$. By Lemma \ref{key0}, it follows that $(u-c_0)_+$ is radial in $B_{R}(a)$, and hence $(u - c_0)_+$ is radial with respect to $a$ in $\R^2$. Therefore, our task now is to show that $(u - c_0)_+$ is locally symmetric. We will show that $(u - c_0)_+$ is locally symmetric in the direction $x_1$. The local symmetry in other directions can be established in a similar manner. Recalling Definition \ref{csts}, let $(u - c_0)_+^t$ denote the CStS of $(u - c_0)_+$ with respect to $x_1$. Thanks to Proposition \ref{pro2}, it suffices to prove that
\begin{equation*}
   \int_{\R^2}|\nabla (u-c_0)_+^t|^2\d x-\int_{\R^2}|\nabla (u-c_0)_+|^2\d x= o(t)
\end{equation*}
as $t\to 0$. By Proposition \ref{pro0} (10), we have
\begin{equation*}
   \int_{\R^2}|\nabla (u-c_0)_+^t|^2\d x-\int_{\R^2}|\nabla (u-c_0)_+|^2\d x\le 0.
\end{equation*}
So it is reduced to proving that
\begin{equation*}
  \int_{\R^2}|\nabla (u-c_0)_+^t|^2\d x-\int_{\R^2}|\nabla (u-c_0)_+|^2\d x\ge o(t)
\end{equation*}
as $t\to 0$. This can be proven by showing that for any $\varepsilon>0$, the following inequality holds:
\begin{equation}\label{3-2}
  \int_{\R^2}|\nabla (u-c_0)_+^t|^2\d x-\int_{\R^2}|\nabla (u-c_0)_+|^2\d x\ge -\varepsilon t
\end{equation}
for all sufficiently small $t$. By Lemma \ref{key1}, \eqref{3-2} can be established by showing that there exists $c_1<c_0$ such that
\begin{equation}\label{3-3}
   \int_{\R^2}|\nabla (u-c_1)_+^t|^2\d x-\int_{\R^2}|\nabla (u-c_1)_+|^2\d x\ge -\varepsilon t
\end{equation}
for all sufficiently small $t$. Next, we focus on demonstrating the existence of such a value of $c_1$. Recalling \eqref{3-1}, let $[(u-c_1)_+^t-(u-c_1)_+]$ be a test function. We obtain
\begin{equation}\label{3-4}
  \int_{\R^2}\nabla u\cdot \nabla (u-c_1)_+^t\d x-\int_{\R^2}\nabla u\cdot \nabla (u-c_1)_+\d x =\sum_{j=1}^{n}\alpha_j\int_{D_j}[(u-c_1)_+^t-(u-c_1)_+]\,\d x,
\end{equation}
where we used the fact that
\begin{equation*}
  \int_{\R^2} (u-c_1)_+^t\d x= \int_{\R^2} (u-c_1)_+\d x.
\end{equation*}
Let $U=\Int(\Im_{c_1})$. Note that the left-hand side of \eqref{3-4} can be estimated as follows:
\begin{equation}\label{3-5}
\begin{split}
    \int_{\R^2}& \nabla u\cdot \nabla (u-c_1)_+^t\d x-\int_{\R^2}\nabla u\cdot \nabla (u-c_1)_+\d x \\
     & =\int_{\R^2}\nabla u\cdot \nabla (u^t-c_1)_+\d x-\int_{\R^2}\nabla u\cdot \nabla (u-c_1)_+\d x \\
     & =\int_{U^t}\nabla u\cdot \nabla u^t\,\d x-\int_{U}\nabla u\cdot \nabla u\, \d x\\
     &\le \frac{1}{2}\int_{U^t}|\nabla u^t|^2\,\d x+\frac{1}{2}\int_{[U]^t}|\nabla u|^2\,\d x-\int_{U}|\nabla u|^2\, \d x\\
     & =\frac{1}{2}\left(\int_{U^t}|\nabla u^t|^2\,\d x-\int_{U}|\nabla u|^2\, \d x\right)\\
     &\ \ \ \ \ \ \ \ \ \ \ +\frac{1}{2}\left(\int_{U^t}|\nabla u|^2\,\d x-\int_{U}|\nabla u|^2\, \d x\right)\\
     &=\frac{1}{2}\left( \int_{\R^2}|\nabla (u-c_1)_+^t|^2\d x-\int_{\R^2}|\nabla (u-c_1)_+|^2\d x\right)\\
     &\ \ \ \ \ \ \ \ \ \ \ +\frac{1}{2}\left(\int_{U^t}|\nabla u|^2\,\d x-\int_{U}|\nabla u|^2\, \d x\right).
\end{split}
\end{equation}
Combining \eqref{3-4} and \eqref{3-5}, we obtian
\begin{equation*}
\begin{split}
     & \int_{\R^2}|\nabla (u-c_1)_+^t|^2\d x-\int_{\R^2}|\nabla (u-c_1)_+|^2\d x \\
     &  \ge -\left(\int_{U^t}|\nabla u|^2\,\d x-\int_{U}|\nabla u|^2\, \d x\right)+2\sum_{j=1}^{n}\alpha_j\int_{D_j}[(u-c_1)_+^t-(u-c_1)_+]\,\d x. \\
\end{split}
\end{equation*}
Therefore, \eqref{3-3} would hold provided that
\begin{equation}\label{3-6}
  \left(\int_{U^t}|\nabla u|^2\,\d x-\int_{U}|\nabla u|^2\, \d x\right)-2\sum_{j=1}^{n}\alpha_j\int_{D_j}[(u-c_1)_+^t-(u-c_1)_+]\,\d x\le \varepsilon t
\end{equation}
for all sufficiently small $t$. To summarize, our current goal is to find a $c_1<c_0$ such that \eqref{3-6} holds. Let $\gamma_k$ be a sequence of decreasing numbers satisfying $\gamma_k <c_0$ for $k = 1, 2, \dots$, and $\lim_{k \to \infty} \gamma_k = -\infty$. We will show that $\gamma_k$ satisfies our requirements if $k$ is sufficiently large. For clarity, we proceed by discussing the following two cases separately.

\smallskip
\textbf{Case 1}: $\Omega=0$. In this case, we have $\omega_0\ge \Omega=0$. Recall that
\begin{equation}\label{3-7}
  \nabla u(x)=\frac{1}{2\pi}\int_{\R^2}\frac{x-y}{|x-y|^2}\omega_0(y)\,\d y=\frac{1}{2\pi}\int_{\R^2}\omega_0(y)\,\d y\cdot \frac{x}{|x|^2}+O\left(\frac{1}{|x|^2}\right)
\end{equation}
as $|x|\to \infty$. By analyzing the radial and angular components of $\nabla u$ in \eqref{3-7}, it can be seen that the level set $\Im_{\gamma_k}$ will approximate a circle as $k\to \infty$. In fact, let $U_k=\Int(\Im_{\gamma_k})$,  $r_{k}=\sup\{r: B_r\subset U_k\}$, and $R_k=\inf \{r: U_k\subset B_r\}$. Clearly, $r_k\to \infty$ as $k\to \infty$. In standard polar coordinates $(\rho, \theta)$, the level set $\Im_{\gamma_k}$ can be parametrized as $\rho=\rho_k(\theta),\ \theta\in[0, 2\pi]$. Using \eqref{3-7}, we check that $\sup_{\theta\in[0, 2\pi]}|\rho_k'(\theta)|\le C$ for some constant $C>0$ independent of $k$. It follows that $R_k=r_k+O(1)$ as $k\to \infty$. Furthermore, we obtain $\mathcal{H}_1(\Im_{\gamma_k})\le 4\pi r_k$ for all sufficiently large $k$. Here, the symbol $\mathcal{H}_1$ denotes the one-dimensional Hausdorff measure.

Now, let us consider \eqref{3-6}, where $c_1$ is replaced by $\gamma_k$. For convenience, we denote $A=\frac{1}{2\pi}\int_{\R^2}\omega_0(y)\,\d y\not=0$. Note that
\begin{equation}\label{3-8}
\begin{split}
    \int_{U_k^t}|\nabla u|^2\,\d x-\int_{U_k}|\nabla u|^2\, \d x & =\int_{U_k^t\backslash U_k}|\nabla u|^2\,\d x-\int_{U_k\backslash U_k^t}|\nabla u|^2\, \d x \\
      =\int_{U_k^t\backslash U_k}&\left(|\nabla u|^2-\frac{A^2}{r_k^2}\right)\,\d x-\int_{U_k\backslash U_k^t}\left(|\nabla u|^2-\frac{A^2}{r_k^2}\right)\, \d x,
\end{split}
\end{equation}
where we used that fact that $\LL^2(U_k^t)=\LL^2(U_k)$. We now proceed to demonstrate that, for sufficiently large values of $k$, the following holds:
\begin{equation}\label{3-9}
 \left|\int_{U_k^t\backslash U_k}\left(|\nabla u|^2-\frac{A^2}{r_k^2}\right)\,\d x-\int_{U_k\backslash U_k^t}\left(|\nabla u|^2-\frac{A^2}{r_k^2}\right)\, \d x\right|\le \frac{\varepsilon}{2}t,\ \ \ \forall\, t\ge 0,
\end{equation}
 In view of \eqref{3-7}, we have
\begin{equation*}
  |\nabla u (x)|^2=\frac{A^2}{|x|^2}+O\left(\frac{1}{|x|^3}\right)
\end{equation*}
as $|x|\to \infty$. It follows that
\begin{equation}\label{3-10}
  |\nabla u (x)|^2-\frac{A^2}{r_k^2}=O\left(\frac{1}{r_k^3}\right),\ \ \ x\in U_k^t\triangle U_k,
\end{equation}
where the symbol $\triangle$ means the symmetric difference. Next, we proceed to estimate $\LL^2(U_k^t\triangle U_k)$. Recalling the definitions \ref{Def2-2} and \ref{csts}, we observe that
\begin{equation*}
  U_k^t\triangle U_k\subset \{x\in \R^2: \text{dist}(x, \Im_{\gamma_k})\le R_k t\}.
\end{equation*}
for all $t\ge0$. In fact, this follows from the observation that the intervals move with a velocity of at most $R_k$ (see Lemma 4.1 on page 3017 of \cite{Gom2021MR4312192} for a similar argument). It follows that  $\LL^2(U_k^t\triangle U_k)\le 2\mathcal{H}_1(\Im_{\gamma_k}) R_kt$; see also Eq.\,(4.2) on page 3018 of \cite{Gom2021MR4312192}. Thus, we obtain $\LL^2(U_k^t\triangle U_k)\le 10\pi r_k^2 t$ for all $t\ge 0$ provided that $k$ is sufficiently large. In conclusion, we have
\begin{equation}\label{3-11}
\begin{split}
     & \left| \int_{U_k^t\backslash U_k}\left(|\nabla u|^2-\frac{A^2}{r_k^2}\right)\,\d x \right|+ \left| \int_{U_k^\backslash U_k^t}\left(|\nabla u|^2-\frac{A^2}{r_k^2}\right)\,\d x \right|  \\
     &\ \ \ \le \frac{C}{r_k^3}\LL^2(U_k^t\triangle U_k)\le 10\pi C r_k^{-1}t
\end{split}
\end{equation}
for any $t\ge 0$ provided that $k$ is sufficiently large. Here, $C>0$ is a constant depending on $\omega_0$, but independent of $k$ and $t$. From \eqref{3-11}, it follows that there exists a sufficiently large $k$ such that \eqref{3-9} holds, and consequently, \eqref{3-8} also holds. Now, let us fix such a value of $k$, whence, by \eqref{3-8}, we have
\begin{equation}\label{3-12}
   \int_{U_k^t}|\nabla u|^2\,\d x-\int_{U_k}|\nabla u|^2\, \d x\le \frac{\varepsilon}{2}t,\ \ \ \forall\, t\ge 0.
\end{equation}
Next, we proceed to estimate the other terms in \eqref{3-6}, with $c_1$ replaced by $\gamma_k$. For each $j\in\{1, \dots, n\}$, we have
\begin{equation}\label{3-13}
\begin{split}
    & \int_{D_j}[(u-\gamma_k)_+^t-(u-\gamma_k)_+]\,\d x \\
     &\ \  =\int_{V^{(j)}_0}[(u-\gamma_k)_+^t-(u-\gamma_k)_+]\,\d x-\sum_{i=1}^{m_j}\int_{\overline{V^{(j)}_i}}[(u-\gamma_k)_+^t-(u-\gamma_k)_+]\,\d x.
\end{split}
\end{equation}
Recall that $V_i^{(j)}=\operatorname{int}(\Gamma^{(j)}_i)$. Furthermore, $(u-\gamma_k)_+=u-\gamma_k$ is weakly superharmonic in $V_i^{(j)}$, and remains constant on $\partial V_i^{(j)}=\Gamma^{(j)}_i$. Applying Lemmas \ref{key2} and \ref{key3}, we obtain
\begin{equation}\label{3-14}
\int_{V^{(j)}_0}[(u-\gamma_k)_+^t-(u-\gamma_k)_+]\,\d x  =o(t),\ \ \  \int_{\overline{V^{(j)}_i}}[(u-\gamma_k)_+^t-(u-\gamma_k)_+]\,\d x=o(t)
\end{equation}
as $t\to 0$, for any $j=1, \dots, n$ and $i=0, \dots, m_j$. Combining \eqref{3-13} and \eqref{3-14}, we get
\begin{equation}\label{3-66}
   \int_{D_j}[(u-\gamma_k)_+^t-(u-\gamma_k)_+]\,\d x=o(t),\ \ \ j=1,\dots, n,
\end{equation}
which implies
\begin{equation}\label{3-15}
 \left| \sum_{j=1}^{n}\alpha_j\int_{D_j}[(u-\gamma_k)_+^t-(u-\gamma_k)_+]\,\d x\right|\le \frac{\varepsilon}{2}t
\end{equation}
for all sufficiently small $t$. Therefore, from \eqref{3-12} and \eqref{3-15}, it follows that \eqref{3-6} holds with $c_1$ replaced by $\gamma_k$. This concludes the proof of this case.

\smallskip
\textbf{Case 2}: $\Omega\not=0$. For convenience, we adopt the notation used in the proof of the previous case. In this case, we have
\begin{equation}\label{3-16}
  \nabla u(x)=\nabla (\mathcal{N} \ast \omega_0)+\Omega\,x=\Omega\, x+\frac{1}{2\pi}\int_{\R^2}\omega_0(y)\,\d y\cdot \frac{x}{|x|^2}+O\left(\frac{1}{|x|^2}\right)
\end{equation}
as $|x|\to \infty$. A careful examination of \eqref{3-16} reveals that, in this case, the radial component of $\nabla u$ plays an even more dominant role. In standard polar coordinates $(\rho, \theta)$, the level set $\Im_{\gamma_k}$ can be parametrized as $\rho=\rho_k(\theta),\ \theta\in[0, 2\pi]$. We check that $\sup_{\theta\in[0, 2\pi]}|\rho_k'(\theta)|\le Cr_k^{-2}$ for some constant $C>0$ independent of $k$. It follows that $R_k=r_k+O\left(\frac{1}{r_k^2}\right)$ as $k\to \infty$. Next, we proceed to further examine the convexity of $U_k=\Int(\Im_{\gamma_k})$. Note that
\begin{equation*}
  \partial_{11}u(x)=\Omega+O\left(\frac{1}{|x|^2} \right),\ \ \ \partial_{12}u(x)=\partial_{21}u(x)=O\left(\frac{1}{|x|^2} \right),\ \ \  \partial_{22}u(x)=\Omega+O\left(\frac{1}{|x|^2} \right)
\end{equation*}
as $|x|\to \infty$. By \cite[Proposition 6.1.7]{Kran1999MR1730695}, it follows that $U_k$ is convex for all sufficiently large $k$. At this point, for every $x_2\in \R$, the set $U_k(x_2)=\{x_1\in \R: (x_1, x_2)\in U_k\}$ is an interval (possibly empty), denoted by $(b_k-\eta_k, b_k+\eta_k)$. Since $R_k=r_k+O\left(r_k^{-2}\right)$, we have $\sup_{x_2\in \R}|b_k|\le Cr_k^{-2}$ for some constant $C>0$ independent of $k$. Recalling the definitions \ref{Def2-2} and \ref{csts}, and noting that the intervals move with a velocity of at most $\sup_{x_2\in \R}|b_k|$, we obtain
\begin{equation*}
  U_k^t\triangle U_k\subset \{x\in \R^2: \text{dist}(x, \Im_{\gamma_k})\le Cr_k^{-2} t\}.
\end{equation*}
for all $t\ge0$. Thus, we have $\LL^2(U_k^t\triangle U_k)\le C\mathcal{H}_1(\Im_{\gamma_k}) r_k^{-2} t\le Cr_k^{-1} t$ for all $t\ge 0$ provided that $k$ is sufficiently large.

Analogous to the previous case, our initial goal is to demonstrate that
\begin{equation}\label{3-18}
   \int_{U_k^t}|\nabla u|^2\,\d x-\int_{U_k}|\nabla u|^2\, \d x\le \frac{\varepsilon}{2}t,\ \ \ \forall\, t\ge 0
\end{equation}
provided that $k$ is large enough. Note that
\begin{equation*}
  |\nabla u(x)|^2=|\nabla (\mathcal{N} \ast \omega_0)|^2+2\Omega x\cdot \nabla (\mathcal{N} \ast \omega_0)+\Omega^2|x|^2.
\end{equation*}
By the Hardy-Littlewood inequality (see Proposition \ref{pro0} (7)), we first have
\begin{equation}\label{3-19}
  \int_{U_k^t}|x|^2\,\d x-\int_{U_k}|x|^2\, \d x\le 0.
\end{equation}
Note that $|\nabla (\mathcal{N} \ast \omega_0)|\le C/|x|$ as $|x|\to \infty$. Thus, we have
\begin{equation}\label{3-20}
\begin{split}
    & \left|\int_{U_k^t}|\nabla (\mathcal{N} \ast \omega_0)|^2\,\d x-\int_{U_k}|\nabla (\mathcal{N} \ast \omega_0)|^2\, \d x\right| \\
     &\ \ =\left|\int_{U_k^t\backslash U_k}|\nabla (\mathcal{N} \ast \omega_0)|^2\,\d x-\int_{U_k\backslash U_k^t}|\nabla (\mathcal{N} \ast \omega_0)|^2\, \d x\right| \\
     &\ \  \le \frac{C}{r_k^2}\LL^2(U_k^t\triangle U_k)\le  \frac{C}{r_k^3}t\le \frac{\varepsilon}{4}t
\end{split}
\end{equation}
for all $t\ge0$, provided that $k$ is sufficiently large. Similarly, noting that $|x\cdot \nabla (\mathcal{N} \ast \omega_0)|\le C$, we have
\begin{equation}\label{3-21}
\begin{split}
    & |2\Omega|\cdot \left| \int_{U_k^t}x\cdot \nabla (\mathcal{N} \ast \omega_0)\,\d x-\int_{U_k}x\cdot \nabla (\mathcal{N} \ast \omega_0)\, \d x\right| \\
     & \ \ =|2\Omega|\cdot \left| \int_{U_k^t\backslash U_k}x\cdot \nabla (\mathcal{N} \ast \omega_0)\,\d x-\int_{U_k\backslash U_k^t}x\cdot \nabla (\mathcal{N} \ast \omega_0)\, \d x\right| \\
     & \ \ \le |2\Omega|\cdot C\LL^2(U_k^t\triangle U_k)\le  \frac{C}{r_k}t\le \frac{\varepsilon}{4}t
\end{split}
\end{equation}
for all $t\ge0$, provided that $k$ is sufficiently large. Combining \eqref{3-19}, \eqref{3-20}, and \eqref{3-21}, we obtain \eqref{3-18}.

Now, let us fix a value of $k$ for which \eqref{3-18} holds. The subsequent task in the proof is to verify that
\begin{equation}\label{3-22}
 \left| \sum_{j=1}^{n}\alpha_j\int_{D_j}[(u-\gamma_k)_+^t-(u-\gamma_k)_+]\,\d x\right|\le \frac{\varepsilon}{2}t
\end{equation}
for all sufficiently small $t$. The proof of this result is identical to that of the previous case and is therefore omitted. It follows from \eqref{3-18} and \eqref{3-22} that \eqref{3-6} holds with $c_1$ replaced by $\gamma_k$. This concludes the proof of this case.

In summary, we have proved that $\omega_0$ must be radially symmetric if $\Omega\le \inf \omega_0/2$.

\subsection{The case $\Omega \geq \sup \omega_0 / 2$}
In this case, $\Omega \ge 0$, and $u = \mathcal{N} \ast \omega_0+\frac{\Omega}{2}|x|^2$ is weakly subharmonic. Let $\tilde{u}=-u$. Then $\tilde{u}$ is weakly superharmonic, and
\begin{align}\label{3-23}
\begin{cases}
-\Delta \tilde{u}=-\sum_{j=1}^n \alpha_j 1_{D_j}+2\Omega \ &  \text{in}\ \R^2,\\
\tilde{u}=-c^{(j)}_i\ &\text{on}\ \Gamma^{(j)}_i,\ \text{for}\ j=1, \dots, n;\ i=0, \dots, m_j.
\end{cases}
\end{align}
 Repeating the proof of the previous case with $u$ replaced by $\tilde{u}$ yields the desired result. Since the details of the proof are essentially the same, they are omitted here. This completes the proof of this case. Hence, $\omega_0$ must also be radially symmetric if $\Omega \geq \sup \omega_0 / 2$.

In summary, we have proved that if a uniformly rotating multi-patch $\omega_0$ with compact support has an angular velocity satisfying $\Omega \leq \inf \omega_0 / 2$ or $\Omega \geq \sup \omega_0 / 2$, then it must be radial. This completes the proof of Theorem \ref{thm}.

\section{Radial symmetry of uniformly rotating smooth solutions} \label{s4}

In this section, we establish the radial symmetry of uniformly rotating smooth solutions to \eqref{1-1}.

Let $\omega_0 \in C^2(\R^2)$ be a compactly supported function. Suppose that $\omega(x, t) = \omega_0(e^{-i\Omega t} x)$ is a uniformly rotating smooth solution of \eqref{1-1}, meaning that it satisfies \eqref{1-5} for some $\Omega \in \mathbb{R}$. Our goal is to show that $\omega_0$ must exhibit radial symmetry if $\Omega\le \inf \omega_0/2$ or $\Omega\ge \sup \omega_0/2$. Here, the term "radial symmetry" is to be understood as radial symmetry with respect to a point, up to a translation, in the stationary case (i.e., $\Omega = 0$).

Observe that the desired radial symmetry follows if we can show that the stream function $u = \mathcal{N} \ast \omega_0 + \frac{\Omega}{2}|x|^2$ is radial, up to a translation in $\mathbb{R}^2$. We now proceed to demonstrate that $u$ possesses this property.

Note that the stream function $u$ satisfies the following problem:
\begin{align}\label{4-1}
\begin{cases}
-\Delta u=\omega_0-2\Omega \ \ \text{in}\ \R^2, & \\
u \ \text{is constant on each connected component of a regular level set of}\ \omega_0 .
\end{cases}
\end{align}
Similar to the case of multi-patch solutions, we consider the two cases $\Omega \leq \inf \omega_0 / 2$ and $\Omega \geq \sup \omega_0 / 2$ separately. For convenience, we adopt the notation used in the proof of the case of multi-patch solutions

\subsection{The case $\Omega\le \inf \omega_0/2$} In this case, $\Omega \ge 0$, and $u = \mathcal{N} \ast \omega_0+\frac{\Omega}{2}|x|^2$ is weakly superharmonic. Similar to the discussion in Subsection \ref{subsec3-1}, the problem reduces to verifying that, for any $\varepsilon > 0$, there exists $c_1 \ll 0$ such that
\begin{equation}\label{4-2}
  \left(\int_{U^t}|\nabla u|^2\,\d x-\int_{U}|\nabla u|^2\, \d x\right)-\int_{\R^2}\omega_0(x)[(u-c_1)_+^t-(u-c_1)_+]\,\d x\le \varepsilon t
\end{equation}
for all sufficiently small $t$. Let $\gamma_k$ be a decreasing sequence such that $\lim_{k \to \infty} \gamma_k = -\infty$. We will demonstrate that $\gamma_k$ satisfies the requirements for sufficiently large $k$. First, by repeating the same arguments as in Subsection \ref{subsec3-1}, we find that there exists a sufficiently large $k$ such that
\begin{equation}\label{4-3}
  \int_{U_k^t}|\nabla u|^2\,\d x-\int_{U_k}|\nabla u|^2\, \d x\le \frac{\varepsilon}{2}t,\ \ \ \forall\, t\ge 0.
\end{equation}
It remains to show that
\begin{equation}\label{4-4}
\mathcal{I}(t):=  \left| \int_{\R^2}\omega_0(x)[(u-\gamma_k)_+^t-(u-\gamma_k)_+]\,\d x\right|\le \frac{\varepsilon}{2}t
\end{equation}
for all sufficiently small $t$. To this aim, we shall employ an approximation technique. For any integer $n>1$, we can approximate $\omega_0$ by a step function $w_n$ of the form $w_n=\sum_{j=1}^{M_n}\alpha_j 1_{D_j}(x)$, which satisfies all the following properties (see, e.g., page 2997 in \cite{Gom2021MR4312192}):
\begin{itemize}
  \item [(a)]Each $D_j$ is a connected open domain with $C^2$ boundary and possibly has a finite number of holes, denoted by $\{V_i^{(j)}\}_{i=0}^{m_j}$;
  \item [(b)]Each connected component of $\partial D_j$ is a connected component of a regular level set of $ \omega_0$;
  \item [(c)]$D_i\cap D_j=\varnothing$ if $i\not=j$;
  \item [(d)]$\|w_n-\omega_0\|_{L^\infty(\R^2)}\le \frac{2}{n}\|\omega_0\|_{L^\infty(\R^2)}$.
\end{itemize}
We estimate $\mathcal{I}(t)$ as follows:
\begin{equation}\label{4-5}
\begin{split}
   \mathcal{I}(t) &   = \left| \int_{\R^2}\omega_0(x)[(u-\gamma_k)_+^t-(u-\gamma_k)_+]\,\d x\right| \\
     & \le \left| \int_{\R^2}\left[\omega_0(x)-w_n(x)\right][(u-\gamma_k)_+^t-(u-\gamma_k)_+]\,\d x\right| \\
    & \ \ \ \ \ \ +\left|\sum_{j=1}^{M_n}\alpha_j\int_{D_j}[(u-\gamma_k)_+^t-(u-\gamma_k)_+]\,\d x \right|.
\end{split}
\end{equation}
Let $R_k=\inf \{r: \text{supp}(u-\gamma_k)_+\subset B_r\}$. Then $\text{supp}(u-\gamma_k)_+^t\subset Q_{R_k}$ for all $t\ge 0$. By Proposition \ref{pro0} (10), we have $\|(u-\gamma_k)_+^t-(u-\gamma_k)_+\|_{L^\infty(\R^2)}\le Ct$ for some constant $C>0$ independent of $t$. We estimate
\begin{equation*}
\begin{split}
     & \left| \int_{\R^2}\left[\omega_0(x)-w_n(x)\right][(u-\gamma_k)_+^t-(u-\gamma_k)_+]\,\d x\right| \\
     & \ \ \ \le \pi R_k^2\|w_n-\omega_0\|_{L^\infty(\R^2)}\|[(u-\gamma_k)_+^t-(u-\gamma_k)_+]\|_{L^\infty(\R^2)}\\
     &\ \ \ \le 2\pi R_k^2 C\|\omega_0\|_{L^\infty(\R^2)}n^{-1}t.
\end{split}
\end{equation*}
Thus, there exists a sufficiently large $n$ such that
\begin{equation}\label{4-6}
  \left| \int_{\R^2}\left[\omega_0(x)-w_n(x)\right][(u-\gamma_k)_+^t-(u-\gamma_k)_+]\,\d x\right|\le \frac{\varepsilon}{4}t,\ \ \ \forall\, t\ge 0.
\end{equation}
We now fix such an $n$. By virtue of \eqref{4-1}, $u$ remains constant on each connected component of $\partial D_j$. Using Lemmas \ref{key2} and \ref{key3}, we can get
\begin{equation}\label{4-7}
  \left|\sum_{j=1}^{M_n}\alpha_j\int_{D_j}[(u-\gamma_k)_+^t-(u-\gamma_k)_+]\,\d x \right|=o(t)
\end{equation}
as $t\to 0$. In fact, this only requires repeating the proof for \eqref{3-66}. Hence, the details are omitted here. Combining \eqref{4-5}, \eqref{4-6}, and \eqref{4-7}, we obtain \eqref{4-4}. The proof of this case is thus complete.

\subsection{The case $\Omega \geq \sup \omega_0 / 2$}

In this case, $\Omega \ge 0$, and $u = \mathcal{N} \ast \omega_0+\frac{\Omega}{2}|x|^2$ is weakly subharmonic. Let us now turn our attention to the function $\tilde{u}=-u$. Then $\tilde{u}$ is weakly superharmonic, and
\begin{align*}
\begin{cases}
-\Delta \tilde{u}=-\omega_0+2\Omega \ \ \text{in}\ \R^2, & \\
\tilde{u} \ \text{is constant on each connected component of a regular level set of}\ \omega_0 .
\end{cases}
\end{align*}
By repeating the proof of the previous case, replacing $u$ with $\tilde{u}$, we obtain the desired result. Since the details of the proof are essentially identical, they are omitted here. This concludes the proof of this case.

In summary, we have shown that if a uniformly rotating smooth solution $\omega_0$ with compact support has an angular velocity satisfying $\Omega \leq \inf \omega_0 / 2$ or $\Omega \geq \sup \omega_0 / 2$, then it must be radial. This concludes the proof of Theorem \ref{thm2}.

\bigskip
\noindent \textbf{Data Availability}: Data sharing is not applicable to this article as no datasets were generated or analyzed during the current study.

\smallskip
\noindent \textbf{Conflict of interest}: The authors declare that they have no conflict of interest.

\appendix

\section{Non-radial uniformly rotating vortex patches with sign-changing vorticity} \label{s5}

In this appendix, we establish a bifurcation result for uniformly rotating vortex patch solutions with sign-changing vorticity. The purpose is to complement Theorem \ref{thm} by showing the sharpness of the admissible interval for the angular velocity.

Consider the vortex patch solution of \eqref{1-1} given by
\be \label{E:Doubly-VP}
\omega =1_{D_1} - \kappa 1_{D_2},
\ee
where $D_2 \subset D_1$ are simply-connected bounded domains,
and $\kappa>1$ is a parameter. When $\kappa=1$, \eqref{E:Doubly-VP} recovers the doubly-connected vortex patches studied in \cite{Hmidi2016,Hmi2016MR3545942,wang2024degeneratebifurcationstwofolddoublyconnected}.

The main unknowns in the patch solution \eqref{E:Doubly-VP} are the vortical domains $D_j$, for $j=1,2$. By the Riemann mapping theorem \cite{Pommerenke1992}, the complement of $D_j$ is parameterized by the conformal mappings
\be \label{E:Conformal}
\Phi_j(z) = a_j z + G_j(z) : \mathbb{C} \setminus B_1 \rightarrow \mathbb{C} \setminus D_j, \; j=1,2,
\ee
where $a_1=1,a_2=b \in (0,1)$, and
\[
G_j(z) = \sum_{n \geq 0} \frac{b_n^j}{z^n}, \quad j=1,2,
\]
describe the shapes of the domains. Since $G_j(\,\cdot \,)$ is holomorphic in the exterior disk $\mathbb{C}\setminus B_1$, this map is unqiuely determined by its trace on the unit circle
\[
g_j(z) :=  G_j(z) \big|_{\mathbb{T}} =\sum_{n \geq 0}  \frac{b_n^j}{z^n},\quad |z|=1.
\]

By the conformal mapping parameterization, the equations of uniformly rotating vortex patches \eqref{E:1-3} is equivalent to the contour equations
\be \label{E:Contour}
\begin{split}
& {\rm Im} \bigg( \left(2 \Omega \overline{\Phi_1(w)} + I_1(\Phi_1(w)) - \kappa I_2(\Phi_1(w)) \right) \Phi^\prime_1(w) w \bigg) =0, \\
	& {\rm Im} \bigg( \left(2 \Omega \overline{\Phi_2(w)} + I_1(\Phi_2(w)) - \kappa I_2(\Phi_2(w)) \right) \Phi^\prime_2(w) w \bigg) =0,
\end{split}
\ee
where
\be \label{E:Cauchy-m}
 I_j(z): = \frac{1}{2\pi i}\int_\mathbb{T} \frac{\overline{z}-\overline{\Phi_j(\xi)}}{z-\Phi_j(\xi)}\Phi_j'(\xi)d \xi, \quad j=1,2
\ee
are Cauchy-type integral operators.

A direct computation demonstrates that concentric disks are the solutions of \eqref{E:Contour} for any $\Omega \in \R$. To construct  non-radial solutions, we employ the Crandall-Rabinowitz bifurcation theorem \cite{Crandall1971}. Here we follow the formulation given by Buffoni and Toland \cite{Buffoni2003}:
\begin{lemma}\cite[Theorem 8.3.1]{Buffoni2003} \label{L:CR}
	Suppose that $X$ and $Y$ are Banach spaces, that $F: X \times \R \rightarrow Y$ is of class $C^k$, $k \geq 2$, and that $F(0,\lambda)=0 \in Y$ for all $\lambda \in \R$. Suppose that
	\begin{enumerate}
		\item $L=\p_x F(0,\lambda_0)$ is a Fredholm operator of index $0$;
		\item ${\rm Ker} L = span(\xi_0), \xi_0 \in X$, $dim {\rm Ker} \,L=1$.
		\item The transversality condition holds:
		\[
		\p_{\lambda,x}^2 F(0,\lambda_0) (\xi_0,1) \notin {\rm Ran}\, L.
		\]
	\end{enumerate}
	Then $(0,\lambda_0)$ is a bifurcation point. More precisely, there exists $\epsilon>0$ and a branch of solution
	\[
	\left\{(x,\lambda)=(s\chi(s),\Lambda(s)), |s| < \epsilon, s \in \R \right\} \subset X \times \R ,
	\]
	such that $\Lambda(0)=\lambda_0; \chi(0)=\xi_0$;
    \[
	F(s\chi(s),\Lambda(s))=0 \text{ for all $|s|<\epsilon$,}
	\]
	\[
	\Lambda \text{ and } s\rightarrow s \chi(s) \text{ are of class $C^{k-1}$, and $\chi$ is of class $C^{k-2}$, on $(-\epsilon,\epsilon)$}
	\]
	There exists and open set $U_0 \subset X \times \R$ such that $(0, \lambda_0) \in U_0$ and
	\[
	\left\{ (x, \lambda) \in U_0: F(x, \lambda)=0, x \neq 0 \right\} = \left\{(s\chi(s), \Lambda(s)), 0<|s|<\epsilon \right\}.
	\]
	In particular, if $F$ is analytic, $\Lambda$ and $\chi$ are analytic functions on $(-\epsilon,\epsilon)$.
\end{lemma}
\begin{remark}
To apply Lemma \ref{L:CR}, it suffices that $\p_x F, \p_\lambda F$ and $\p_{x \lambda}^2 F$ are continuous.
\end{remark}

We recall the functional setting introduced in \cite{Hmidi2016} for analyzing the bifurcation of uniformly rotating vortex patches. Let $\alpha>0$. By the symmetry,  we consider uniformly rotating vortex patches which the vortical domain are parameterized by functions from the space
\[
C_{real}^{1+\alpha}(\mathbb{T}) :=\left\{ f\in C^{1+\alpha}(\mathbb{T}) \; \big| \; f(w) = \sum_{n \geq 0}  \frac{a_n}{w^n}, \,a_n \in \R, |w|=1 \right\},
\]
providing a smooth conformal parametrization of the vortex patch boundaries. Define the function space $X:=C_{real}^{1+\alpha}(\mathbb{T}) \times C_{real}^{1+\alpha}(\mathbb{T})$. We consider a small neighbourhood $0 \in V:= B(0,\epsilon_0) \subset X$ serving as the space of perturbations of the boundaries of the inner and outer vortical domains.

For the target space, we let $Y:=H \times H$, where
\[
H:= \left\{ h \in C^\alpha(\mathbb{T}) \; \big|\; h(\theta) = \sum_{n \geq 1} b_n \sin n \theta, b_n \in \R \right\}.
\]
This space ensures that the equations governing the evolution of the vortex patch boundaries remain well-posed in a function space capturing their nontrivial deformations. This functional framework allows us to formulate the bifurcation problem as a nonlinear mapping
\[
F(\lambda,g_1,g_2)=(F_1,F_2): X \times \R \rightarrow Y,
\]
where  $\lambda: = 1-2\Omega$ and
\[
\begin{split}
	& F_1(\lambda,g_1,g_2):={\rm Im} \big\{ \left((1-\lambda) \overline{\Phi_1(w)} + I_1(\Phi_1(w)) - \kappa I_2(\Phi_1(w)) \right) \Phi^\prime_1(w) w \big\}, \\
	& F_2(\lambda,g_1,g_2): ={\rm Im} \big\{ \left((1-\lambda) \overline{\Phi_2(w)} + I_1(\Phi_2(w)) - \kappa I_2(\Phi_2(w)) \right) \Phi^\prime_2(w) w \big\},
\end{split}
\]
encodes the governing contour equations \eqref{E:Contour}. To apply Lemma \ref{L:CR}, it suffices to verify conditions (1)–(3).

First, we have that $F:X \times \R \rightarrow Y$ is a smooth map, i.e., $\p_{g} F,\p_\lambda F, \p_{g \lambda}^2 F$ are continuous, where $g=(g_1,g_2) \in X$. The proof follows arguments analogous  to those in \cite{Hmidi2016}, which we omit here for brevity.

Next, we analyze the linearization of $F$ on the trivial solutions. Let $h(w) := \sum_{n \geq 0}  \frac{\alpha_n}{{w}^n}$, $k(w)= \sum_{n \geq 0} \frac{\beta_n}{{w}^n}$. Analogously to the computation in \cite[Section 7]{Hmidi2016}, the linearized operator is given by
\be
{D} F(\lambda,0,0)(h,k)(w) = \begin{pmatrix}
	L_\lambda^{1}(h,k)(w) \\
	L_\lambda^{2}(h,k)(w)
\end{pmatrix},
\ee
where
\[
\begin{split}
L_\lambda^{1}(h,k)(w) & = {\rm Im} \left(1(1-\lambda+\kappa b^2) \overline{w}h(w) + (\kappa b^2 -\lambda)h^\prime(w) + \kappa b^2 \overline{w} k(\frac{w}{b}) \right), \\
L_\lambda^{2}(h,k)(w) & = b\, {\rm Im} \left(- \overline{w} h(\frac{w}{b}) + \lambda \overline{w} k(w) + (\kappa -\lambda)k^\prime(w) \right).
\end{split}
\]
Expressing the lineazied operator $DF(\lambda,0,0)$ in terms of Fourier multipliers, we get
\[
DF(\lambda,0,0) \begin{pmatrix} \alpha_ n \\
	\beta_n
	\end{pmatrix} \bar{w}^n = M_n  \begin{pmatrix} \alpha_ n \\
	\beta_n
	\end{pmatrix} \sin (n+1) \theta,
\]
with the matrix multiplier
\[
M_n:= \begin{pmatrix} 1-\lambda + \kappa b^2 + n(kb^2-\lambda) & - \kappa b^{n+2} \\
	b^{n+1} & b(n(\kappa-\lambda)-\lambda)
\end{pmatrix}.
\]
The operator $DF(\lambda,0,0)$ is an invertible Fredholm operator with index $0$ if and only if $M_n$ is invertible for all $n \geq 0$, and $\lambda \notin \{\kappa b^2, \kappa \}$.

On the other hand, the linearized operator $DF(\lambda,0,0)$ has non-trivial kernel if
\be \label{E:multi}
\Delta_n(\lambda):= \lambda^2 - \frac{(n+1)\kappa b^2 + n\kappa +1}{n+1} \lambda + \frac{\kappa b^{2n+2} + n \kappa +n(n+1)b^2 \kappa^2}{(n+1)^2} = 0,
\ee
for some $n \geq 0$, and the corresponding eigenvector is given by
\be \label{E:Eigen-vec}
v_0:=\left((n(\kappa-\lambda)-\lambda)\bar{w}^n, - b^n \bar{w}^n\right).
\ee

Equation \eqref{E:multi} yields a pair of distinct eignvalues to $DF(\lambda,0,0)$
\be \label{E:lambda}
\lambda_n^\pm: = \frac{\kappa b^2}{2} +  \frac{n \kappa +1}{2(n+1)}  \pm  \sqrt{\bigg(\frac{\kappa b^2}{2} - \frac{n\kappa-1}{2(n+1)}\bigg)^2 - \frac{\kappa b^{2n+2}}{(n+1)^2} },
\ee
if
\[
S_n:=\bigg(\frac{\kappa b^2}{2} - \frac{n\kappa-1}{2(n+1)}\bigg)^2 - \frac{\kappa b^{2n+2}}{(n+1)^2} >0, \quad n \geq 0.
\]

When $n=0$, we have
\[
\lambda_0^+ = \begin{cases}
	\kappa b^2 & \;\; b > \frac{1}{\sqrt{\kappa}} \\
	1 & \;\; b <  \frac{1}{\sqrt{\kappa}}
\end{cases}, \quad \lambda_0^-  = \begin{cases}
1 & \;\; b > \frac{1}{\sqrt{\kappa}} \\
\kappa b^2 & \;\; b <  \frac{1}{\sqrt{\kappa}}
\end{cases}.
\]
This result indicates that the precise values of $\kappa$ and $b$ could affect the spectral structure essentially. In particular, assume there exists $n >0$ such that $\lambda_n^\pm =  \kappa b^2$, we get
\[
\bigg(\frac{\kappa b^2}{2} - \frac{n\kappa-1}{2(n+1)}\bigg)^2 - \frac{\kappa b^{2n+2}}{(n+1)^2} = \left(\frac{\kappa b^2}{2} - \frac{n\kappa+1}{2(n+1)} \right)^2.
\]
That implies
\[
b^{2n+2} - (n+1)b^2 + n=0.
\]
There exists a unique $b_n \in (0,1)$ such that $\lambda_n^\pm = \kappa b^2$. Therefore, for $b \in (0,1) \setminus \{b_n\}_{n \geq 1}$, $\lambda_0^+ (\text{ or } \lambda_0^-) = \kappa b^2$ is a simple eigenvalue of $DF(\lambda,0,0)$.

In contrast to the case $\kappa=1$ where the eigenvalues $\lambda_n^\pm$ and $\lambda_m^\pm$ are distinct for $n \neq m$.
When $\kappa>1$, there may exist $n \neq m \geq 0$ such that $\{\lambda_n^+,\lambda_n^-\} \cap \{\lambda_m^+,\lambda_m^-\} \neq \emptyset$, as demonstrated for $n=0$ above. Namely,
 $\lambda_n^\pm$ may not be a simple eigenvalue even when $S_n>0$. To avoid complex high dimensional bifurcations, we adjust the parameter $b$ here.

Define the intersection set
\[
\Theta_{n,m}:=\{ b\in (0,1) \big| \{\lambda_n^+,\lambda_n^-\} \cap \{\lambda_m^+,\lambda_m^-\} \neq \emptyset, \; 0 \leq n \neq m \},
\]
and let
\be
\Theta:= \bigcup_{n \neq m \geq 0}\Theta_{n,m}.
\ee
For nonempty $\Theta_{n,m} \neq \emptyset$, the eigenvalue collision condition yields:
\[
\frac{1-\kappa}{2} \left( \frac{1}{n+1}- \frac{1}{m+1} \right) = \pm \sqrt{S_n} \mp \sqrt{S_m}.
\]
Squaring both sides, we get the polynomial equation
\be \label{E:S_n}
\left((\frac{1-\kappa}{2})^2 \left( \frac{1}{n+1}- \frac{1}{m+1} \right)^2 - S_n - S_m\right)^2 = 4 S_n S_m.
\ee
For given $\kappa>1,n \neq m \geq 0$, equation \eqref{E:S_n} constitutes a polynomial in $b$ of degree $4(\max(n,m)+1)$. Thus $\Theta_{n,m}$ contains  at most $4(\max(n,m)+1)$ isolated points, then $\Theta$ is a countable set.

Moreover, for each $n \geq 0$, solving equation $S_n=0$ produces at most $n+1$ distinct $b$. Define $S:= \cup_{ n\geq 0} \left\{b \in (0,1)\; \big| \; S_n = 0 \right\}$
For $b \in (0,1)\setminus (\Theta \cup S)$, all eigenvalue $\lambda_{n}^\pm$ of $DF(\lambda,0,0)$ are simple, ensuring the applicability of Lemma \ref{L:CR}.

Finaly, we examine the transverality condition. One can explicitly compute the derivatives of $F$ as follows:
\[
\p_\lambda D F(g_1,g_2,\lambda)(h,k) = \begin{pmatrix}
	-{\rm Im} \left( \overline{h(w)} \Phi_1^\prime(w) w + \overline{\Phi_1(w)} h^\prime(w) w \right) \\
	- {\rm Im} \left( \overline{k(w)} \Phi_2^\prime(w) w + \overline{\Phi_2(w)} k^\prime(w) w \right)
\end{pmatrix}.
\]
Thus we get
\[
\p_\lambda D F(\lambda,0,0)(h,k)=\begin{pmatrix}
	-{\rm Im} \left( \overline{h(w)} w + h^\prime(w) \right) \\
	-{\rm Im} \left( \overline{k(w)} w + k^\prime(w) \right)
\end{pmatrix}.
\]
Combining the eigenvector $\hat{v}_0$ given in \eqref{E:Eigen-vec} leads
\[
\p_{\lambda} DF(\lambda,0,0)(\hat{v}_0) = -(n+1) \begin{pmatrix}
	n(\kappa-\lambda) - \lambda \\
	-b^{n+1}
\end{pmatrix} \sin(n+1)\theta.
\]
If $\p_{\lambda} DF(\lambda,0,0)(\hat{v}_0) \in Ran(DF(\lambda,0,0))$, then
\[
\begin{pmatrix}
	n(\kappa-\lambda) - \lambda \\
	-b^{n+1}
\end{pmatrix}
\]
is a scalar of one column of the matrix $M_n$, that is,
\be
\left(n(\kappa-\lambda)-\lambda\right)^2 = \kappa b^{2n+2}.
\ee
Combining equation \eqref{E:multi}, we get
\[
\left(n(\kappa-\lambda)-\lambda \right) \left(1-\lambda + \kappa b^2 + n(\kappa b^2 -\lambda) \right) + \left(n(\kappa -\lambda) - \lambda \right)^2 =0.
\]
Thus either $\lambda = \frac{n}{n+1} \kappa$, or $\lambda = \frac{\kappa b^2}{2} + \frac{n \kappa +1}{2(n+1)}$. We claim that such $\lambda$ do not coincide with the eigenvalues $\lambda_n^\pm$. Clearly, $\lambda_n^\pm \neq \frac{\kappa b^2}{2} + \frac{n \kappa +1}{2(n+1)}$ when $S_n>0$. It suffices to demonstrate $\lambda_n^\pm  \neq \frac{n}{n+1} \kappa$. If not, we would get
\be
\frac{\kappa b^2}{2} + \frac{n\kappa +1}{2(n+1)} \pm  \sqrt{\bigg(\frac{\kappa b^2}{2} - \frac{n\kappa-1}{2(n+1)}\bigg)^2 - \frac{\kappa b^{2n+2}}{(n+1)^2} } = \frac{n}{n+1} \kappa.
\ee
That is impossible since otherwise one would has
\[
\bigg(\frac{\kappa b^2}{2} - \frac{n\kappa-1}{2(n+1)}\bigg)^2 - \frac{\kappa b^{2n+2}}{(n+1)^2} = \left(\frac{n \kappa-1}{2(n+1)} - \frac{\kappa b^2}{2} \right)^2.
\]
Therefore, for $b \in (0,1) \setminus \Theta \cup S$, $\{\lambda_n^\pm\}_{n \geq 0}$ are simple eigenvalues of the linearizied operator $DF(\lambda,0,0)$.


Now let us consider the functional space
\[
C_{real,m}^{1+\alpha}(\mathbb{T}) :=\left\{ f\in C^{1+\alpha}(\mathbb{T}) \; \big| \; f(w) = \sum_{n \geq 1}  \frac{a_{n(m+1)-1}}{w^{n(m+1)-1}}, \,a_n \in \R, |w|=1 \right\}.
\]
For any $f \in C_{real,m}^{1+\alpha}(\mathbb{T})$, the conformal mapping
\[
\Phi_j(z) = a_j z + \sum_{n \geq 1} \frac{a_{n(m+1)-1}}{z^{n(m+1)-1}} = a_j z \left(  +  a_j^{-1} \sum_{n \geq 1}  \frac{a_{n(m+1)-1}}{z^{n(m+1)}} \right)
\]
parameterizes the domain with $m+1$-fold symmetry, which yields
\[
\Phi_j(e^{\frac{2\pi i}{m}}z) = e^{\frac{2\pi i}{m}} \Phi_j(z), \l;\; j=1,2.
\]
The nonlinear map $F$ is equivariant under the rotation action $e^{\frac{2\pi i}{m}} \cdot z$ thus map $F:X_m \times \R \rightarrow Y_m$ is well-defined, where $X_m:= C_{real,m}^{1+\alpha}(\mathbb{T}) \times C_{real,m}^{1+\alpha}(\mathbb{T})$ and $Y_m:=H_m \times H_m$,
\[
H_m:= \left\{ h \in C^\alpha(\mathbb{T}) \; \big|\; h(\theta) = \sum_{n \geq 1} b_{n(m+1)} \sin n(m+1) \theta,\; b_{n(m+1)} \in \R \right\}.
\]
Apply Lemma \ref{L:CR} on the nonlinear map equation $0= F(\lambda,g_1,g_2):X_m \times \R  \rightarrow Y_m$, we establish that
\begin{theorem}
Given $b \in (0,1)\setminus (\Theta \cup S)$ and $\kappa>1$. Let $m \geq 1$ satisfy
\[
\left(\frac{m\kappa-1}{2(m+1)}-\frac{\kappa b^2}{2} \right)^2 - \frac{\kappa b^{2(m+1)}}{(m+1)^2}>0.
\]
There exists a curve of non-radial $m+1$-fold symmetric vortex patch solutions
\[
\omega= 1_{D_1} - \kappa 1_{D_2}
\]
bifurcating from the concentric disks $\omega = 1_{B_1} - \kappa 1_{B_b}$ at  the angular velocity
\[
\Omega_m^{\pm}:=  \frac{1-\kappa b^2}{4} + \frac{1-\kappa}{4} + \frac{\kappa-1}{4(m+1)} \mp \frac{1}{2} \sqrt{\bigg(\frac{\kappa b^2}{2} - \frac{m\kappa-1}{2(m+1)}\bigg)^2 - \frac{\kappa b^{2m+2}}{(m+1)^2} }.
\]
In particular, there exists $1$-fold symmetric vortex patch solutions
\[
\omega= 1_{D_1} - \kappa 1_{D_2}
\]
bifurcating from the concentric disks $\omega = 1_{B_1} - \kappa 1_{B_b}$ at  the angular velocity at the angular velocity $\Omega = \frac{1-\kappa b^2}{2}$.
\end{theorem}

Let $m \rightarrow \infty$, we get
\[
\Omega_m^+ \rightarrow \frac{1-\kappa b^2}{2}, \quad \Omega_m^{-} \rightarrow \frac{1-\kappa}{2}.
\]
The latter exactly coincides the threshold $\inf \omega /2$ appears in Theorem \ref{thm}, and the former converges to another bound $\sup \omega /2$ as $b \rightarrow 0$.








\end{document}